\documentclass[10pt,reqno]{amsart}

\usepackage[utf8]{inputenc} 

\usepackage[margin=1in]{geometry} 

\usepackage{graphicx} 
\usepackage{float} 

\usepackage{changepage}

 \usepackage[parfill]{parskip} 
 
\usepackage{booktabs} 
\usepackage{array} 
\usepackage{paralist} 
\usepackage{verbatim} 
\usepackage{subfig} 
\usepackage{mathrsfs}
\usepackage{amssymb}
\usepackage{xcolor}
\usepackage{amsthm}
\usepackage{amsmath,amsfonts,amssymb,esint,hyperref,varioref}
\usepackage[noabbrev, capitalize]{cleveref}

\usepackage{autonum}

\usepackage{graphics,color}
\usepackage{enumerate, enumitem}
\usepackage{mathtools,centernot}
\usepackage{cases}
\usepackage{amsrefs}
\usepackage{bbm}
\usepackage{xfrac}

\usepackage{bookmark}

\newtheorem{theorem}{Theorem}[section]
\newtheorem{lemma}[theorem]{Lemma}

\newtheorem{corollary}[theorem]{Corollary}
\newtheorem{definition}[theorem]{Definition}
\newtheorem{proposition}[theorem]{Proposition}
\newtheorem{remark}[theorem]{Remark}

\numberwithin{equation}{section} 

\newcommand{\norm}[1]{\left\|#1\right\|}
\newcommand{\abs}[1]{\left|#1\right|}

\newcommand*{\R}{\ensuremath{\mathbb{R}}}

\newcommand*{\N}{\ensuremath{\mathbb{N}}}

\newcommand{\eps}{\varepsilon}
\newcommand*{\tr}{\ensuremath{\mathrm{tr\,}}}

\newcommand{\e}{\varepsilon}

\newcommand{\quotes}[1]{``#1''}
\newcommand{\dd}{\,\mathrm{d}}
\newcommand{\loc}{\mathrm{loc}}
\newcommand{\pt}{\partial_t}

\renewcommand{\phi}{\varphi}

\renewcommand{\MR}[1]{} 

\usepackage{color, graphicx}
\usepackage{mathrsfs, dsfont}

\usepackage[]{hyperref}
\hypersetup{
    colorlinks=true,       
    linkcolor=red,          
    citecolor=blue,        
    filecolor=red,      
    urlcolor=cyan           
}

\def\dist{\mathop{\rm dist}\nolimits}    
\def\div{\mathop{\rm div}\nolimits}    
\def\supp{\mathop{\rm supp}\nolimits}    
 
\def\tr{\mathop{\rm Tr}\nolimits}

  \DeclareMathOperator{\trace}{tr}

\newcommand{\be}{\begin{equation}}
\newcommand{\ee}{\end{equation}}

\title{Inhomogeneous incompressible Euler with codimension $1$ singular structures}

\author[M. Inversi]{Marco Inversi}
\address[M. Inversi]{Departement Mathematik und Informatik, Universit\"at Basel, CH-4051 Basel, Switzerland}
\email{marco.inversi@unibas.ch}

\author[A. Violini]{Alessandro Violini}
\address[A. Violini]{Departement Mathematik und Informatik, Universit\"at Basel, CH-4051 Basel, Switzerland}
\email{alessandro.violini@unibas.ch}

\date{\today}

\thanks{\textit{Acknowledgements}.
 The authors acknowledge the support of the SNF grant FLUTURA: Fluids, Turbulence, Advection No. 212573. A. Violini is supported by the Deutsche Forschungsgemeinschaft project “Inhomogeneous and compressible fluids: statistical solutions
and dissipative anomalies” within SPP 2410 Hyperbolic Balance Laws in Fluid Mechanics: Complexity,
Scales, Randomness (CoScaRa).}

\begin{document}

\begin{abstract}
This paper is concerned with the inhomogeneous incompressible Euler system. We establish a Duchon--Robert type approximation theorem for the distribution describing the local energy flux of bounded solutions. The velocity field is assumed to have bounded variation or bounded deformation with respect to the spatial variable. The density satisfies no-vacuum condition and has $BV$ regularity in space, allowing for a system made by two immiscible fluids separated by a Lipschitz hypersurface. By means of a careful analysis of the traces along hypersufaces of bounded vector fields with measure divergence, we show that the dissipation does not give mass to hypersurfaces of codimension one, even if the velocity field, the density and the pressure have jumps. This feature is specific of incompressible models. As a consequence, we show that the dissipation measure vanishes if it is concentrated on a countably rectifiable set of space-time codimension one, even if this is dense. For instance, in the simplest case of a system made by two incompressible immiscible fluids with $BV$ velocity, the dissipation measure vanishes as soon as the interface between the two fluids has (space-time) finite perimeter.  
\end{abstract}

\maketitle

\section{Introduction}

We consider the inhomogeneous incompressible Euler system 
\begin{equation}\label{IE} \tag{E}
\left\{\begin{array}{l}
\partial_t (\rho u) + \div( \rho u  \otimes u ) + \nabla p= f  \\
\partial_t \rho + \div(\rho u) = 0 \\ 
\div u=0
\end{array}\right. \qquad \text{ in } \Omega \times (0,T)
\end{equation}
where $\Omega = \mathbb{T}^d, \R^d$ or any open subset of the Euclidean space. The system \eqref{IE} describes the conservation of momentum and mass in an incompressible ideal fluid system with non constant density under the action of the external force field $f$. We will mainly focus on the analysis in the interior of $\Omega \times (0,T)$. Thus, the boundary conditions an the initial data are not given. If not otherwise specified, the function spaces used in this note are tacitly intended to be in $\Omega, (0,T)$ or $\Omega \times (0,T)$. This choice is made only for convenience. 

\subsection{A brief background} 
A classical computation shows that smooth solutions to \eqref{IE} satisfy the local energy balance 
\begin{equation} \label{eq: local energy balance}
\partial_t \left(\rho \frac{\abs{u}^2}{2}\right) +\div  \left(\left( \rho \frac{\abs{u}^2}{2}+p\right)u\right) = - D[u] + f\cdot u \qquad \text{in } \Omega \times (0,T), 
\end{equation}
with $D[u] \equiv 0$. However, for irregular solutions to \eqref{IE}, it is expected that \eqref{eq: local energy balance} holds with a non-trivial distribution $D[u]$ describing the failure of the exact local energy balance. In the homogeneous case, i.e. when $\rho \equiv 1$, the study of rough solutions to \eqref{IE} is related to understanding some fundamental features of turbulence, such as the appearance of anomalous dissipation phenomena. Indeed, in 1949 Onsager proposed to look at \eqref{IE} as a universal model for turbulent flows \cite{O49}, giving a deterministic point of view to the Kolmogorov's theory of turbulence (see the monograph by Frisch \cite{Frisch95} for an extensive presentation). In the homogeneous setting, Onsager predicted that solutions to \eqref{IE} possessing \emph{more than $\sfrac{1}{3}$ of derivative} should conserve the kinetic energy. In the last decades, these predictions have been validated by the mathematical research. In modern terms \cite{DR00}, when $\rho \equiv 1$, given $u \in L^{3}_{x,t}$ and $f \in L^{\sfrac{3}{2}}_{x,t}$, then $D[u]$ in \eqref{eq: local energy balance} is a well-defined distribution named after Duchon--Robert and $D[u] = \lim_{\e \to 0} D_\e[u]$, where\footnote{Here, the limit is intended in $\mathcal{D}'_{x,t}$ and $\eta$ is any smooth even convolution kernel, supported in $B_1$ and with unit mass and $x$ is picked in a compact set such that $x + \e z \in \Omega$. }
\begin{equation} \label{eq: DR formula}
    D_\e[u] : = \frac{1}{4} \int_{B_1} \nabla \eta(z) \cdot \frac{\delta_{\e z} u_t(x)}{\e} \abs{\delta_{\e z} u_t(x)}^2 \dd z , \qquad \delta_{\e z} u_t(x) := u(x+\e z, t) - u(x,t). 
\end{equation}
In \cites{Ey94, CET94} it is shown that $D_\e [u] \to 0$ strongly in $L^1_{x,t}$ as soon as $u \in L^3_t B^{\theta}_{3, \infty}$ for some $\theta > \sfrac{1}{3}$. On the other side, for $\theta<\sfrac{1}{3}$ convex integration methods \cites{DS13,Is18,BDSV19,GR24} give solutions in $C^\theta$ with non-trivial distribution $D[u]$. Subsequent improvements \cites{Is22,DK22} also provided H\"older continuous solutions with $D[u]>0$, which eventually have been extended to the whole Onsager supercritical range in the $L^3$-based framework \cites{GKN24_1,GKN24_2}. Although these results mathematically validate the Onsager prediction, the critical case $\theta = \sfrac{1}{3}$ remains open. For solutions to the homogeneous \eqref{IE} with this critical regularity, the sequence $\{ D_{\e}[u]\}_{\e}$ defined by \eqref{eq: DR formula} is merely bounded in $L^1_{x,t}$ and it is not known if it converges to $0$ (weakly or strongly). In this direction, energy conservation in $L^3 B^{\sfrac{1}{3}}_{3,\infty}$ is known under some additional assumptions preventing large oscillations of the velocity field at small scales (\cites{CCFS08, FW18,BGSTW19}). We refer to the companion paper \cite{DRINnew} and the reference therein for a detailed discussion on this principle. 

Recently, energy conservation in the homogeneous case for solutions in $L^\infty_{x,t} \cap L^1_t BV_x$ has been established in \cite{DRINV23}. This result is qualitatively different from the classical energy conservation ones for several reasons. First of all, in \cite{DRINV23} it is shown that $D_\e[u] \rightharpoonup 0$ for solutions in a class where the Duchon--Robert approximation \eqref{eq: DR formula} is bounded in $L^1_{x,t}$ and the strong $L^1$ convergence may fail. Moreover, $L^\infty_{x,t} \cap L^1_t BV_x$ embeds sharply into $L^3_t B^{\sfrac{1}{3}}_{3,\infty}$, in the sense that $BV_x \cap L^\infty_x$ does not embed into the closure of smooth functions with respect to the $B^{\sfrac{1}{3}}_{3, \infty}$ norm (see \cite{DRINnew}*{Theorem 1.2}). To the best of our knowledge, this is the first result available in such a critical setting. Unlikely the proofs in \cites{DR00, CET94, CCFS08, FW18, BGSTW19} that apply simultaneously to compressible and incompressible models, this argument fails in the compressible setting, as confirmed by the  formation of shock waves in 1D Burgers (see the discussion after \cref{thm: no dissipation}). Building on a remarkable idea by Ambrosio \cite{Ambr04}, the analysis in \cite{DRINV23} relies on the structure of $BV$ incompressible vector field, where increments at small scales are controlled by the distributional gradient even if they are not small. The formula \eqref{eq: DR formula} shows the explicit dependence of the approximation $D_\e[u]$ on the convolution kernel $\eta$. Since the Duchon--Robert measure is independent on $\eta$, this choice can be optimized. By standard measure theory, it is proved that  
\begin{equation}
    \abs{D[u]} \lesssim \left( \inf_{\eta} \int_{B_1} \abs{\nabla \eta(z)\cdot M(x,t) z} \dd z \right) \nu, \qquad \text{where } M(x,t) = \frac{d \nabla u_t}{\abs{d \nabla u_t}}(x), \quad \nu = \abs{\nabla u_t} \otimes d t,
\end{equation}
where the infimum is computed among all the smooth even convolution kernels with compact support and unit mass. A clever computation by Alberti \cite{C09}*{Lemma 2.6.2} shows that the infimum above equals $\tr M_{x,t}$. Thus, by the divergence-free condition we conclude that $D[u] \equiv 0$. Particular solutions to homogeneous \eqref{IE} in $L^\infty_{x,t} \cap L^1_t BV_x$ are the Vortex-Sheets, that is solutions whose vorticity is concentrated on a regular hypersurface. This case has been investigated by Shvydkoy \cite{S09}, who proved energy conservation by a completely different argument. In a nutshell, Vortex-Sheets may dissipate kinetic energy only on the hypersurfaces where the vorticity is concentrated, since they are regular elsewhere. However, if the time evolution of the sheets is regular enough, by the dynamical properties of the incompressible model, it is shown that the Duchon--Robert measure does not give mass to the hypersurface where the velocity and the pressure are allowed to jump. This gives a completely different point of view on the dissipation mechanism, that has been generalized by \cite{DRINnew} in the homogeneous setting. 

In the inhomogeneous case, the situation is more complicated, since a non-trivial distribution $D[u]$ in \eqref{eq: local energy balance} may arise due to irregularity both in the velocity field and in the density. To the best of our knowledge, the analogous of the Kolmogorov and Onsager theory for non-homogeneous fluids has not been developed, but it is reasonable to conjecture that only a limited amount of regularity of $\rho, u, p$ is preserved by the vanishing viscosity limit. Therefore, irregular solutions to \eqref{IE} should not satisfy the exact energy balance. We refer to \cites{CY19, LS16, FGSW17} for some results on energy conservation for \eqref{IE}. In these cases, under different assumptions, the proofs are based on commutator estimates in Besov spaces inspired by \cites{CET94, DR00, CCFS08} for the homogeneous case. Our analysis aims to extend the result in \cites{DRINV23, DRINnew} to the inhomogeneous case. Unlike the previous energy conservation results, we point out that our analysis covers the case of a system made by two immiscible inviscid fluids that are separated by a hypersurface, i.e. a system where the density takes two values and it jumps on a set of codimension one. 

\subsection{Main results} Weak solutions to \eqref{IE} are defined in the usual distributional sense. 

\begin{definition} \label{d: weak solution}
We say that $(\rho, u, p) \in L^\infty_{x,t}$ solves the inhomogeneous incompressible Euler equations \eqref{IE} in $\Omega \times (0,T)$ with force $f \in L^1_{x,t}$ if it holds 
\begin{align} \label{eq: momentum}
    \int_0^T \int_\Omega \left[ \rho u \cdot \pt \varphi  +   \rho u \otimes u : \nabla  \varphi  +  p \div  \varphi  +  f \cdot \varphi \right] \dd x \dd t = 0 & \qquad \forall \varphi \in C^\infty_c(\Omega \times (0,T); \R^d), 
    \\ \int_0^T \int_\Omega \rho \left[  \partial_t \varphi + u \cdot \nabla \varphi \right] \dd x \dd t = 0 & \qquad \forall \varphi \in C^\infty_c(\Omega \times (0,T)), \label{eq: transport density}
    \\ \int_\Omega u_t(x) \cdot \nabla \varphi(x) \dd x = 0 & \qquad \forall \varphi \in C^\infty_c(\Omega) \quad \text{for a.e. } t \in (0,T). 
\end{align}
\end{definition}
By a plain density argument, the distributional formulation holds with bounded Lipschitz continuous test function vanishing at the (space-time or spatial) boundary. Throughout this note we will always assume that $\rho$ satisfies the \emph{no-vacuum condition}
\begin{equation} \label{no-vacuum}
    \rho_t(x) \geq C >0 \qquad \text{for a.e. } (x,t) \in \Omega \times (0,T). 
\end{equation}
Since $\rho$ is formally advected by $u$, this condition is preserved along the time evolution in the smooth setting. In the simplest scenario, the density takes two values, corresponding to a system of two immiscible incompressible inviscid fluids that are separated by a hypersurface. To describe such a model, it seems reasonable to consider densities with spatial $BV$ regularity, since this allows for jumps on sets of codimension one in space. We consider a velocity field with spatial $BV$ or $BD$ regularity, which corresponds to the case in which (some) spatial increments scale linearly in $L^1$ (see \cref{ss: BV and BD} for precise definitions). Roughly speaking (see \cref{l: char of BV-BD}), $u$ has bounded variations if and only if 
$$\norm{ u(\cdot + \ell y) - u(\cdot)}_{L^1} \sim \ell \qquad \forall \abs{y} =1, \, \ell >0, $$
while $u$ with bounded deformation is characterized by a control of the \emph{longitudinal} increment, i.e. 
$$\norm{ y \cdot( u(\cdot + \ell y) - u(\cdot))}_{L^1} \sim \ell \qquad \forall \abs{y} =1, \, \ell >0. $$ 
Motivated by the results in \cites{DRINV23, DRINnew} for the homogeneous case, we study the concentration of the Duchon--Robert measure associated to a velocity field and a density with the critical regularity described above (see \cref{ss: curves of measures} and \cref{ss: BV and BD} for precise definitions).

\begin{theorem} \label{thm: DR approximation}
Let $(\rho,u, p) \in L^\infty_{x,t}$ be a weak solution to \eqref{IE} with force $f \in L^1_{x,t}$ according to \cref{d: weak solution} and let $D[u] $ be the Duchon--Robert distribution defined by \eqref{eq: local energy balance}. Assume that $\rho \in L^1_t BV_x$ and $\rho$ satisfies the no-vacuum condition \eqref{no-vacuum}\footnote{Here, we denote by $\nabla^s \rho_t$ the singular part of $\nabla \rho_t$ with respect to $\mathcal{L}^{d}$. }. 
\begin{enumerate}
    \item[(i)] If $u \in L^1_t BV_x$, then $D[u]$ is a Radon measure such that 
    \begin{equation} \label{eq: upper bound DR in BV}
        \abs{D[u]} \ll \abs{\nabla^s \rho_t} \otimes \dd t . 
    \end{equation}
    \item[(ii)] If $u \in L^1_t BD_x$, then $D[u]$ is a Radon measure such that\footnote{Here, we denote by $E^s u_t$ the singular part of the symmetric gradient of $u_t$ with respect to $\mathcal{L}^d$. } 
    \begin{equation} \label{eq: upper bound DR in BD}
        \abs{D[u]} \ll  \abs{\nabla^s \rho_t} \otimes \dd t + \abs{E^s u_t} \otimes \dd t. 
    \end{equation} 
\end{enumerate}
\end{theorem} 
The above theorem is a particular case of the more general \cref{thm: DR approximation sharp}. The proof provides an explicit formula for the defect measure $D[u]$. Since it is rather long, we did not included it in the statement. If $\rho \in L^1_t W^{1,1}_x$, our computation shows 
\begin{equation} \label{eq: DR formula simplified}
    D[u] = \lim_{\e \to 0}  \frac{\rho}{2 \rho_\e} [u \otimes u : \nabla (\rho u)_\e - u \cdot \div((\rho u \otimes u)_\e) ] \qquad \text{in } \mathcal{D}'_{x,t},
\end{equation}
which reduces to the Duchon--Robert formula in the case $\rho \equiv 1$ after some manipulations. If the velocity field has bounded variation, we show that the dissipation is concentrated on the jump set of the density (see $(i)$ in \cref{thm: DR approximation}). In particular, if $\rho \in L^1_t W^{1,1}_x$, by \eqref{eq: upper bound DR in BV} we infer that $D[u] \equiv 0$, coherently with the energy conservation result in the homogeneous case \cite{DRINV23}*{Theorem 1.2}. The first statement of \cref{thm: DR approximation} is done by a classical mollification type argument and a careful analysis of the error terms produced along the computations. The oscillations of the velocity field do not produce dissipation, as shown by the first author in \cite{DRINV23} for the homogeneous case. The concentration on the singular set of the velocity field is ruled out by choosing appropriately the convolution kernel in the mollification process, an idea of Ambrosio \cite{Ambr04}. As in the homogeneous case, the incompressiblity of the velocity field plays a key role in the optimization procedure. 

Then we investigate the case of velocity fields with bounded deformation (see $(ii)$ in \cref{thm: DR approximation}). We show that the dissipation is concentrated on the support of the singular part of the gradient of the density and the singular part of the symmetric gradient of the velocity field. In this case, we point out that the concentration on the singular part of the symmetric gradient of the velocity cannot be ruled out, since the strategy of \cite{DRINV23} is not available. 

The second main result of this paper concerns with the accumulation of $D[u]$ on space-time hypersurfaces of codimension one (see \cref{ss: traces}). This is the analogous of \cite{DRINnew}*{Theorem 1.3} in the homogeneous case. We recall that a set $\Sigma \subset \R^d$ is said to be countably $\mathcal{H}^{d-1}$-rectifiable if we can write 
$$\Sigma = \mathcal N \cup \bigcup_{i \in \N} \Sigma_i, $$
where $\mathcal{H}^{d-1}(\mathcal{N})=0$ and $\Sigma_i \subset \R^d$ are (at most countably many) Lipschitz embedded hypersurfaces.

\begin{theorem} \label{thm: no dissipation} 
Let $(\rho, u, p) \in L^\infty_{x,t}$ be a weak solution to \eqref{IE} with $f \in L^1_{x,t}$ according to \cref{d: weak solution} and let $D[u]$ be the Duchon--Robert distribution defined by \eqref{eq: local energy balance}. Assume that $D[u]$ is a finite Borel measure on $\Omega \times (0,T)$ and $(\rho, u, p)$ have bilateral traces at any Lipschitz space-time hypersurface in the sense of \cref{d: inner/outer trace}. Then, for any countably $\mathcal{H}^d$-rectifiable set $\Sigma$ it holds that $\abs{D[u]}(\Sigma) = 0$. 
\end{theorem}

We show that, even if $u,\rho, p$ are allowed to jump on a space-time hypersurface $\Sigma$ of codimension one, then the dissipation does not accumulate on $\Sigma$. We point out that this feature is specific of incompressible models. It is well-known that there exist bounded solutions to the 1D Burgers equation with $BV$ regularity, whose dissipation set is a space-time line\footnote{ We recall that, in his original paper \cite{Burg48}, the author proposed 1D Burgers equation as a simplified model to capture the qualitative properties of hydrodynamical turbulence, such as the formation of dissipation layers, i.e. hypersurfaces that \textit{``play an important part in the energy exchange, as they represent the main regions where energy is dissipated''} \cite{Burg48}*{Section I}. }. In the proof of \cref{thm: no dissipation}, by the analysis of the normal traces \cites{DRINV23, CDRINVN24, DRINnew}, we exploit the dynamics of the Euler equations and the structure of the left hand side in \eqref{eq: local energy balance} as a space-time divergence. In the stationary case we show that particles that intersect a hypersurface $\Sigma$ tangentially do not dissipate energy due to the incompressibility condition, while particles that intersect $\Sigma$ transversally do not experience an instantaneous impulse for dynamical reasons. We refer to \cite{DRINnew}*{Section 1.1} for a description of the heuristic argument in the homogeneous stationary case, that can be adapted in our context (see also \cref{r:final comments stationary}).

As a corollary, we conclude that the Duchon--Robert measure vanishes identically whenever it is concentrated on a countably $\mathcal{H}^d$-rectifiable set. For completeness, we state the following result (see \cite{DRINnew}*{Corollary 1.4} for the analogous result in the homogeneous case). 

\begin{corollary} \label{cor: no dissipation BV/BD} 
Let $(\rho,u, p) \in L^\infty_{x,t}$ be a weak solution to \eqref{IE} with force $f \in L^1_{x,t}$ according to \cref{d: weak solution} and let $D[u] $ be the Duchon--Robert distribution defined by \eqref{eq: local energy balance}. Suppose that $\rho \in L^1_t BV_x \cap SBV_{x,t}$, $p$ has bilateral traces at any oriented Lipschitz space-time hypersurface in $\Omega \times (0,T)$ and either $u \in L^1_t BV_x$ or $u \in SBD_{x,t}$ in the sense of \cref{d: BD x-t}. Then, $D[u] \equiv 0$. 
\end{corollary}

The latter applies to the case of a system made by two immiscible incompressible inviscid fluids with $BV$ spatial regularity. We show that the dissipation is ruled out if the contact surface stays of space-time finite perimeter (see also \cref{r: final comments 2 fluids}). This is coherent with the Rayleigh--Taylor instability, where the instantaneous formation of tentacles and small vortices destroys completely the geometry of the interface (see \cite{gebhard2021new} and the reference therein). 

\begin{remark} \label{r: space-time BD BV}
We briefly comment on the $BD$ assumption. By interpolation and \cref{l: char of BV-BD}, \cref{L:Bd_gradient}, $u \in L^\infty \cap BD$ satisfies 
\begin{equation}
\norm{ y \cdot( u(\cdot + \ell y) - u(\cdot))}_{L^p} \sim \ell^{\frac{1}{p}} \qquad \forall \abs{y} =1, \, \ell >0 \, p \in [1, +\infty).  
\end{equation}
To the best of our knowledge, this is the first result available in the inhomogeneous case under a control on the longitudinal structure function of the velocity field. The space-time $SBD$ regularity for the velocity and $SBV$ for the density are needed for technical reasons. It would be natural to investigate the case of a velocity field in $L^1_t SBD_x$ and a density in $L^1_t SBV_x$. Thus, we would have that for any time slice $u_t, \rho_t$ jump on a countably $\mathcal{H}^{d-1}$-rectifiable set $\Sigma_t \subset \Omega$. However, in this setting, the time evolution of the sets $\Sigma_t$ could be completely wild. In some sense, the space-time assumptions allow to \quotes{glue} the space hypersurfaces $\Sigma_t$ in a space-time countably $\mathcal{H}^d$-rectifiable hypersurface $\Sigma$, preventing from a very rough time evolution of these sheets. We also point out that $BV_{x,t}, BD_{x,t}$ are functional classes where the space-time traces required by \cref{cor: no dissipation BV/BD} exist (see \cref{T:trace_in_BV-BD}). In the stationary case, the picture simplifies considerably. For instance, by \cref{cor: no dissipation BV/BD} we prove that $D[u]$ vanishes whenever the $\rho \in BV$ and $u \in BD$, provided that both the gradient of the density and the symmetric gradient of the velocity have no \quotes{Cantor} part (see \cref{ss: BV and BD} for precise definitions). 
\end{remark}

For completeness, we state the global energy balance for weak solutions of \eqref{IE} on bounded domains with Lipschitz boundary. In this case, the system \eqref{IE} is coupled with the \emph{no-flux boundary condition} 
\begin{equation} \label{eq: impermeability condition}
   \int_{\Omega} u_t(x) \cdot \nabla \phi(x) \dd x = 0 \qquad \forall \phi \in C^\infty_c(\R^d), \, \text{ for a.e. } t \in (0,T). 
\end{equation}
In the smooth setting, by the divergence theorem, \eqref{eq: impermeability condition} is equivalent to require that $u_t$ is divergence-free in $\Omega$ and it is tangential to $\partial \Omega$. The following result is similar to \cite{DRINV23}*{Theorem 1.3} in the homogeneous setting.  

\begin{corollary} \label{cor: global conservation}
Let $\Omega$ be a bounded open set with Lipschitz boundary and let $(u,\rho, p) \in L^\infty_{x,t}$ be a weak solution to \eqref{IE} in $\Omega \times (0,T)$ with force $f \in L^1_{x,t}$ according to \cref{d: weak solution}. Assume that 
\begin{enumerate}
    \item[(i)] the local energy balance \eqref{eq: local energy balance} holds with $D[u] \equiv 0$ in $\mathcal{D}'_{x,t}$; 
    \item[(ii)] for a.e. $t$ then $u_t$ has zero Lebesgue normal boundary trace at $\partial \Omega$ according to \cref{d: inner/outer trace}.  
\end{enumerate}
Then, the global energy balance holds, i.e. 
\begin{equation} \label{eq: global energy balance}
    \frac{d}{dt} \int_{\Omega} \frac{1}{2} \rho \abs{u}^2 \dd x = \int_{\Omega} f \cdot u \dd x \qquad \text{ in } \mathcal{D}'((0,T)). 
\end{equation}
\end{corollary}

We remark that $u_t$ has zero Lebesgue boundary trace if $u_t \in BD(\Omega)$ and $u_t$ satisfies \eqref{eq: impermeability condition} (see \cref{T:trace_in_BV-BD}).

\subsection{Plan of the paper} 
The paper is organized as follows. To keep this note self-contained, in \cref{s: tools} we collect some technical tools mostly from \cites{EG15, AFP00, DRINV23, CDRINVN24, DRINnew}, such as properties of curves of measures, $BV$ and $BD$ functions and different notion of traces. In \cref{s: proofs} we discuss in details the proofs of our results. 

\section{Tools} \label{s: tools}

\subsection{Time-dependent curves of measures} \label{ss: curves of measures}

Throughout this section, let $\Omega \subset \R^d$ be an open set. We denote by $\mathcal{M}(\Omega; \R^m)$ the space of finite Borel measures with values in $\R^m$. Given $\mu \in \mathcal{M}(\Omega; \R^m)$, we denote by $\abs{\mu} \in \mathcal{M}(\Omega)$ its variation, i.e. the non-negative finite Borel measure defined by 
\begin{equation} \label{eq: variation of a measure}
\langle \abs{\mu}, \varphi \rangle := \sup_{\psi \in C^0(\Omega; \R^m), \abs{\psi} \leq \varphi} \int_\Omega \psi \cdot \, d  \mu \qquad \forall \varphi \in C^0_c (\Omega), \, \varphi \geq 0.
\end{equation}
We then  set $\norm{\mu}_{\mathcal{M}(\Omega)}: = \abs{\mu}(\Omega)$. If the dimension $m$ of the target space is clear from the context, we simply write $\mathcal{M}(\Omega)$ instead of $\mathcal{M}(\Omega; \R^m)$. We will often use the shorthand notation $\mathcal{M}_x$.

\begin{definition} \label{d: curve of measures}
We say that $ \mu = \{\mu_t\}_{t \in (0,T)}$ is a weakly measurable curves of measures if for almost every $t\in I$ we have $\mu_t \in \mathcal{M}(\Omega)$ and for any test function $\phi \in C^0_c(\Omega)$ the map $t \mapsto \int_{\Omega} \phi(x) \dd \mu_t(x)$ is measurable. We say that $\mu \in L^1((0,T); \mathcal{M}(\Omega))$ if $\mu$ is weakly measurable and the map $t \mapsto \norm{\mu_t}_{\mathcal{M}_x}$ is in $L^1((0,T))$. We define the norm $\norm{\mu}_{L^1_t \mathcal{M}_x} : = \norm{\norm{\mu_t}_{\mathcal{M}(\Omega)}}_{L^1(0,T)}$.
\end{definition}

In the following lemmas, we show that any $\mu  \in L^1_t \mathcal{M}_x$ can be identified with a finite space-time Borel measure which admits a disintegration with respect to the Lebesgue measure on $(0,T)$. We also show that this identification holds at the level of the absolutely continuous part and the singular part.  

\begin{lemma} \label{l: curve of measure 1}
Let $\mu = \{ \mu_t\}_{t} \in L^1_t \mathcal{M}_x$. Then linear functional 
\begin{align} \label{eq: joint measure}
    \langle \mu_t \otimes \dd  t, \varphi\rangle : = \int_0^T \int_{\Omega} \varphi(x,t) \dd \mu_t(x) \dd  t  \qquad \forall \varphi \in C^0_c(\Omega \times (0,T))
 \end{align} 
is a finite Borel measure on $\Omega \times (0,T)$ and satisfies $(\pi_t)_{\#} \mu \ll \mathcal{L}^1 \llcorner (0,T)$, where $\pi_t: \Omega \times (0,T) \to (0,T)$ is the projection. Viceversa, let $\mu \in \mathcal{M}_{t,x}$ such that $(\pi_t)_{\#} \mu \ll \mathcal{L}^1 \llcorner (0,T)$. Then there exists $\{ \mu_t\}_{t} \in L^1_t \mathcal{M}_x$ such that $\mu = \mu_t \otimes \dd  t$. Moreover the family $\{ \mu_t\}_{t}$ is uniquely determined for a.e. $t$. 
\end{lemma}

\begin{proof}
We check that for $\phi \in C^0_c(\Omega \times (0,T))$, the map $t \mapsto \int_\Omega \phi(x,t) \dd \mu_t(x) $ is measurable. By \cref{d: curve of measures}, it is immediate to see that this is the case whenever $\phi(x,t) = \sum_{i=1}^n \phi_i(x) \mathds{1}_{[t_{i-1}, t_i)}(t)$, where $\phi_1, \dots, \phi_n \in C^0_c(\Omega)$ and $t_0 < t_1 < \dots < t_n$ is a partition of $(0,T)$. Since such functions approximate uniformly any function in $C^0_c(\Omega \times (0,T))$, for any $\phi \in C^0_c(\Omega \times (0,T))$ the map $t\mapsto \int_{\Omega} \phi(x,t) \dd \mu_t(x)$ can be written as the pointwise limit of a sequence of measurable functions, and it is therefore measurable. Then, we estimate
\begin{equation}
    \abs{\langle \mu_t \otimes \dd t, \phi \rangle} \leq \norm{\phi}_{C^0} \int_0^T \abs{\mu_t}(\Omega) \dd  t = \norm{\phi}_{C^0} \norm{\mu}_{L^1_t \mathcal{M}_x} 
\end{equation}
and we conclude that $\mu_t \otimes \dd t \in \mathcal{M}_{x,t}$. Viceversa, given $\mu \in \mathcal{M}_{x,t}$ such that $(\pi_t)_{\#} \mu \ll \mathcal{L}^1 \llcorner (0,T)$, by the disintegration theorem (see \cite{AFP00}*{Theorem 2.28}) we find $\{ \mu_t \}_{t} \in L^1_t \mathcal{M}_x$ such that $\mu = \, \mu_t \otimes \dd t$. The family $\{\mu_t\}_t$ is uniquely determined at almost every $t$. 
\end{proof}

\begin{lemma} \label{l: curve of measure 2}
Let $\mu = \{ \mu_t\}_{t} \in L^1_t\mathcal{M}_x$. Then $\abs{\mu} = \{ \abs{\mu_t}\}_{t}$ is in $L^1_t\mathcal{M}_x$ and $\norm{\abs{\mu}}_{L^1_t \mathcal{M}_x} = \norm{\mu}_{L^1_t \mathcal{M}_x}$. For any $t$, let $\mu_t = \mu_t^a + \mu_t^s$ be the Radon--Nikodym decomposition of $\mu_t$ with respect to $\mathcal{L}^d$, i.e. $\mu^a_t \ll \mathcal{L}^d$ and $\mu^s_t \perp \mathcal{L}^d$. Then $\{\mu^{a,s}_t\}_{t } \in L^1_t\mathcal{M}_x $ and, letting $\mu = \mu^a + \mu^s$ be the Radon--Nikodym decomposition of $\mu$ with respect to $\mathcal{L}^{d+1}$, it holds that $\mu^{a,s} = \mu^{a,s}_t \otimes \dd  t$ and $\norm{\mu^{a,s}}_{L^1_t\mathcal{M}_x} \leq \norm{\mu}_{L^1_t\mathcal{M}_x}$. Moreover, there exists $f \in L^1_{x,t}$ such that $f(x,t) \mathcal L^d = \mu^a_t$ for a.e. $t$. 
\end{lemma}
 
\begin{proof}
We start by checking that $\{\abs{\mu_t}\}_t$ is weakly measurable. Then, the equality of norms follows immediately. Indeed, for any $t$ and for any $\varphi \in C^0_c(\Omega)$ non-negative we have that 
\begin{equation}
    \langle \abs{\mu_t}, \varphi \rangle = \sup_{\psi \in C^0_c(O), \abs{\psi} \leq \varphi } \int_{\Omega} \psi \dd  \mu_t.   
\end{equation}
Since $C^0_c(\Omega)$ is separable, the supremum can be taken over a countable dense set thus proving that $t \mapsto \langle \abs{\mu_t}, \varphi \rangle$ is measurable. For the second part of the statement, denoting by $\mu_t^{\pm}$ the positive and negative part of $\mu_t$ respectively, we deduce that both $\{ \mu_t^{\pm}\}_{t\in I}$ are in $L^1_t\mathcal{M}_x$ since 
$$\mu_t^\pm = \frac{ \abs{\mu_t} \pm \mu_t }{2}.$$ 
Therefore, without loss of generality, we can assume that $\mu_t$ is nonnegative for any $t$. Since $\mu = \mu_t \otimes \dd  t \in \mathcal{M}_{t,x}$ by \cref{l: curve of measure 1}, the Radon--Nikodym theorem implies $\mu = \mu^a + \mu^s = f(x,t) \mathcal{L}^{d+1} + \mu^s$, 
where $f \in L^1_{x,t}$ and $\mu^s$ is concentrated on a Borel set $A \subset \Omega \times (0,T)$ such that $\mathcal{L}^{d+1}(A) = 0$. Since $\mu$ is non-negative, both $\mu^{a,s}$ are non-negative. By Fubini's theorem we write 
\begin{equation} \label{eq: disintegration of singular part}
    \mu^s = \left( \mu_t - f(x,t) \mathcal{L}^d \right) \otimes \dd  t =: (\mu^s)_t \otimes \dd  t. 
\end{equation}
For any $t$ we set $A_t := A \cap \{ z=t\}$ and we notice that $\mathcal{L}^d(A_t) =0 $ for a.e. $t$. Let us prove that $(\mu^s)_t$ is concentrated on $A_t$ for a.e. $t$, i.e. $(\mu^s)_t((A_t)^c) = 0$ for a.e. $t$. By \eqref{eq: disintegration of singular part} we have that $\{(\mu^s)_t\}_t$ is a disintegration of $\mu^s$ with respect to $\mathcal{L}^1 \llcorner (0,T)$. Since $\mu^s$ is non-negative, by the uniqueness of the disintegration (see \cite{AFP00}*{Theorem 2.28}) we infer that $(\mu^s)_t$ is non-negative for a.e. $t$. Therefore, by \eqref{eq: disintegration of singular part} we infer that 
$$0= \mu^s(A^c) = \int_0^T (\mu^s)_t((A_t)^c) \dd t, $$
yielding to $(\mu^s)_t((A_t)^c) = 0$ for a.e. $t$. To summarize, we have written $\mu_t \otimes d t = \mu = (f(x,t) \mathcal{L}^d + (\mu^s)_t) \otimes d t$. Then, the uniqueness of the disintegration of measures implies $\mu_t = f(x,t) \mathcal{L}^d + (\mu^s)_t$ for a.e $t$.
Recalling that the decomposition of $\mu_t$ by Radon--Nikodym with respect to Lebesgue measure is unique and since we checked that $(\mu^s)_t \perp \mathcal{L}^d$, we infer that 
$$\mu_t^a = f(x,t) \mathcal{L}^d \qquad \text{and}  \qquad \mu^s_t = (\mu^s)_t \qquad \text{ for a.e. } t.$$
In particular, $\{\mu^a_t\}_{t}$ and $\{\mu^s_t\}_{t}$ are weakly measurable in the sense of \cref{d: curve of measures}. The inequalities of norms $\norm{\mu^{a,s}}_{L^1_t\mathcal{M}_x} \leq \norm{\mu}_{L^1_t\mathcal{M}_x}$ follow immediately.
\end{proof}

\subsection{Bounded variation, bounded deformation and difference quotients} \label{ss: BV and BD}

We denote by $BV(\Omega)$ the space of functions with bounded variations, that is functions $f$ whose distributional gradient is a finite Borel measure on $\Omega$, i.e.
$$ BV(\Omega):=\left\{f\in L^1(\Omega)\,:\, \nabla f\in \mathcal{M}(\Omega;\R^d) \right\}. $$
The latter definition generalizes to vector fields $f:\Omega\rightarrow \R^d$. We refer to the monograph \cite{AFP00} for a extensive description of the properties of $BV$ functions. The space of functions with bounded deformation is given by vector fields whose symmetric gradient is a finite Borel measure, i.e.
$$ BD(\Omega;\R^d):=\left\{f\in L^1(\Omega;\R^d)\,:\, E f := \frac{\nabla f+\nabla f^T}{2}\in \mathcal{M}(\Omega;\R^{d\times d}) \right\}. $$
We refer to \cites{T83, ACDalM97} for a detailed description of $BD$ vector fields. We say that $f \in BV_{\rm loc}(\Omega)$ ($f \in BD_{\rm loc}(\Omega)$) if $f \in BV(O)$ ($f \in BD(O)$, respectively) for any open set $O \subset \joinrel \subset \Omega$. Functions of bounded variations and vector fields with bounded deformation are characterized by suitable control of the difference quotients in $L^1$. More precisely, given $f: \Omega \to \R^m$ we denote by 
\begin{align}\label{not:Incr}
    \delta_{y} f (x) := f(x+ y) - f(x ) \qquad \abs{y}< \dist(x, \Omega^c).  
\end{align}
For the reader convenience, we state and prove the following lemma. 

\begin{lemma} \label{l: char of BV-BD} 
Let $\Omega \subset \R^d$ be an open set and $f \in L^1(\Omega; \R^m)$ such that for any open set $A \subset \joinrel \subset \Omega$ it holds 
\begin{equation} \label{eq: char BV}
    \sup_{\e < \dist(A, \Omega^c)} \norm{ \e^{-1}\delta_{\e y} f}_{L^1(A)} < +\infty \qquad \forall y \in B_1.  
\end{equation}
Then, $f \in BV_{\rm loc}(\Omega; \R^m)$. Similarly, assume that $f \in L^1(\Omega; \R^d)$ and for any open set $A \subset \joinrel \subset \Omega$ it holds  
\begin{equation} \label{eq: char BD}
    \sup_{\e < \dist(A, \Omega^c)} \norm{ \e^{-1} y \cdot \delta_{\e y} f}_{L^1(A)} < +\infty \qquad \forall y \in B_1. 
\end{equation}
Then, $f \in BD_{\rm loc}(\Omega; \R^d)$. 
\end{lemma}

\begin{proof}
Assume that \eqref{eq: char BV} holds. Then, for any test function $\phi \in C^\infty_c(\Omega)$ we write 
\begin{align}
    \abs{\langle \phi, \partial_i f_j \rangle} & = \abs{ \int_{\Omega} \partial_i \phi(x) f_j(x) \dd x } = \abs{ \lim_{\e \to 0} \int_{\Omega} \frac{\delta_{\e e_i} \phi(x)}{\e} f_j(x) \dd x } = \abs{ \lim_{\e \to 0}  \int_\Omega\frac{\delta_{-\e e_i} f(x)}{\e} \phi(x) \dd x } \leq C \norm{\phi}_{C^0(\Omega)}. 
\end{align}
Thus, $\partial_i f_j \in \mathcal{M}_{\rm loc}(\Omega)$ by Riesz theorem. Similarly, if \eqref{eq: char BD} holds, for any test function $\phi \in C^\infty_c(\Omega)$ we write 
\begin{align}
    \abs{\langle \phi, \partial_i f_j + \partial_j f_i \rangle} & = \abs{ \int_{\Omega} [\partial_i \phi(x) f_j(x) + \partial_j \phi(x) f_i(x) ]  \dd x } = \abs{ \lim_{\e \to 0} \int_{\Omega}  \frac{\delta_{-\e e_i} f_j(x) + \delta_{-\e e_j} f_i(x)}{\e} \phi(x) \dd x }. 
\end{align}
After some trivial computations, we see 
\begin{equation}
    \delta_{-\e e_i} f_j(x) + \delta_{-\e e_j} f_i(x) = \delta_{-\e (e_i+e_j)} f(x) \cdot (e_i+ e_j) - \delta_{-\e e_i} f(x - \e e_j) \cdot e_i - \delta_{-\e e_j} f(x - \e e_i) \cdot e_j. 
\end{equation}
Therefore, if \eqref{eq: char BD} holds, we conclude 
\begin{equation}
    \abs{\langle \phi, \partial_if_j + \partial_j f_i \rangle} \leq C \norm{\phi}_{C^0(\Omega)}. 
\end{equation}
Thus, $(E f)_{i,j} \in \mathcal{M}_{\rm loc}(\Omega)$ by Riesz theorem. 
\end{proof}

We recall that the difference quotients of $BV$ or $BD$ functions can be canonically decomposed (see \cite{DRINV23}*{Lemma $2.2$}, \cite{Ambr04}*{Theorem $2.4$},\cite{ACM05}*{Lemma 2.4, Proposition 2.5}). 

\begin{lemma}[$BV$ different quotients]\label{L:Bv_gradient}
Let $\Omega \subset \R^d$ be an open set and $f\in BV(\Omega;\R^m)$. We have  
\begin{equation}\label{BV_est_increment}
\norm{\delta_{\eps y}f}_{L^1(A)}\leq  \e \abs{ \nabla f \cdot y} \left( \overline{(A)_\eps}\right) \qquad \forall A \subset \joinrel \subset \Omega \text{ Borel} \quad \forall \e < \dist(A, \Omega^c) \qquad \forall y \in B_1,
\end{equation}
where $\nabla f \cdot z$ is the $m$-dimensional measure, whose $m$-th component is given by $(\nabla f \cdot z )_m:=z_i \partial_i f_m$. Moreover the function $\delta_{\eps y}f$ can be decomposed in $\delta^a_{\eps y}f+\delta^s_{\eps y}f$ where for every $K \subset \Omega$ compact we have
\begin{align}\label{eq:DQstrongconv}
    \lim_{\e \to 0} \norm{\e^{-1} \delta^a_{\eps y} f  -\nabla^a f \cdot y }_{L^1(K)} = 0 \qquad \forall y \in B_1, 
\end{align}
\begin{align}\label{eq:DQweakb}
    \limsup_{\eps \to 0} \norm{\e^{-1} \delta^s_{\eps y}f}_{L^1(K)} \leq \abs{y} \abs{\nabla^s f}(K) \qquad \forall y \in B_1,
\end{align} 
\begin{equation} \label{eq: DQuniform bound}
    \norm{\delta^a_{\e y} f}_{L^1(K)} + \norm{\delta^s_{\e y}f}_{L^1(K)} \leq \e \abs{y} \abs{\nabla f}\left(\overline{(K)_{\e}}\right) \quad \forall y \in B_1 \qquad \forall \e < \dist(K; \Omega^c). 
\end{equation}
\end{lemma}

\begin{lemma}[$BD$ different quotients]\label{L:Bd_gradient}
Let $\Omega \subset \R^d$ be an open set and $f\in BD(\Omega;\R^d)$. 
The function $\delta_{\eps y}f$ can be decomposed in $\delta^a_{\eps y}f+\delta^s_{\eps y}f$ where for every $K \subset \Omega$ compact we have
\begin{align}\label{eq:DQstrongconv in BD}
   \lim_{\e \to 0} \norm{\e^{-1} y \cdot \delta^a_{\eps y}f  -E^a f : y \otimes y }_{L^1(K)} = 0 \qquad \forall y \in B_1, 
\end{align}
\begin{align}\label{eq:DQweakb in BD}
    \limsup_{\eps \to 0} \norm{\e^{-1} y \cdot \delta^s_{\eps y}f}_{L^1(K)} \leq \abs{y}^2 \abs{E^s f}(K) \qquad \forall y \in B_1,
\end{align} 
\begin{equation} \label{eq: DQuniform bound in BD}
     \norm{y \cdot \delta^a_{\e y} f}_{L^1(K)} + \norm{y \cdot \delta^s_{\e y}f}_{L^1(K)} \leq  \e \abs{y}^2 \abs{E f}\left(\overline{(K)_{\e}}\right) \qquad \forall y \in B_1 \quad \forall \e < \dist(A, \Omega^c). 
\end{equation}
\end{lemma}

The gradient of functions with bounded variations and the symmetric gradient of vector fields with bounded deformation satisfy strong structural properties. We recall some notation from \cite{AFP00}*{Chapter 3.6}. Given an open set $\Omega \subset \R^d$ and a function $f: \Omega \to \R^m$, we denote by $\mathcal{S}_f$ the \emph{approximate discontinuity set}, i.e. 
\begin{equation}
    \Omega \setminus \mathcal{S}_f : = \left\{ x \in \Omega \colon \exists \Tilde{f}(x) \in \R^m \text{ s.t. } \lim_{r \to 0} \fint_{B_r(x)} \abs{f(y) - \Tilde{f}(x)} \dd y = 0 \right\}.  
\end{equation} 
For any $x \in \Omega \setminus \mathcal{S}_f$, the value of $\Tilde{f}(x)$ is uniquely determined and it is known as the approximate limit of $f$ at $x$. We denote by $\mathcal{J}_f \subset \mathcal{S}_f$ the \emph{jump set} of $f$, i.e. 
\begin{align} \label{eq: jump set}
    \mathcal{J}_f : =  
        \left\{  x \in \mathcal{S}_f \colon  \exists \nu(x) \in \mathbb{S}^{d-1}, f^+(x) \neq f^-(x) \in \R^m \text{ s.t. } \lim_{r \to 0} \fint_{B_r^{\nu(x)_\pm}(x)} \abs{f(y) - f^\pm(x)} \dd y = 0 \right\},  
\end{align}
where we define  
$$B_r^{\nu_\pm}(x) : =  \{ y \in B_r(x) \colon \langle y-x, \pm \nu \rangle \geq 0 \}. $$ We say that $f^\pm$ are the \emph{one-sided Lebesgue limits} of $f$ at $x$ with respect to $\nu(x)$. For any $x \in \mathcal{J}_f$ the triple $(\nu(x), f^+(x), f^-(x))$ is uniquely determined, up to change of sign of $\nu(x)$ and permutation of $f^+(x)$ with $f^-(x)$. In particular, the tensor product $(f^+(x) - f^-(x))\otimes \nu(x)$ is well defined for $\mathcal{H}^{d-1}$-a.e. $x \in \mathcal{J}_f$. 

To define an orientation on a countably $\mathcal{H}^d$-rectifiable set $\Sigma$, choose pairwise disjoint Borel sets $E_i$ and oriented hypersurfaces $\Gamma_i$ such that $E_i \subset \Gamma_i$ and $\bigcup_{i \in \N} \Gamma_i$ covers $\mathcal{H}^{d-1}$-almost all of $\Sigma$. Then, define 
$$n_\Sigma := n_{\Gamma_i} \text{ on } E_i.$$
This definition depends on the choice of the decomposition, but only up to a sign, since for any pair of Lipschitz hypersurfaces $\Gamma, \Gamma'$ it holds $n_{\Gamma'} = \pm n_{\Gamma}$ for $\mathcal{H}^{d-1}$-a.e. $x \in \Gamma \cap \Gamma'$. 

We recall some properties of the gradient of $BV$ functions (see \cite{AFP00}*{Chapter 3.9}). The same result holds for the symmetric gradient of $BD$ vector fields (see \cite{ACDalM97}*{Remark 4.2, Theorem 4.3, Proposition 4.4}).

\begin{theorem} \label{thm structure of the gradient BV/BD}
Given $f \in BV(\Omega)$ or $f \in BD(\Omega)$, then the jump set $\mathcal{J}_f$ defined by \eqref{eq: jump set} is countably $\mathcal{H}^{d-1}$-rectifiable (oriented by a unit normal). If $f \in BV(\Omega)$, let $\nabla f = \nabla^a f + \nabla^s f$ be the Radon--Nikodym decomposition of $\nabla f$ with respect to the Lebesgue measure. Then, $\nabla^s f$ admits the following decomposition in \quotes{jump part} and \quotes{Cantor part} 
$$ \nabla^s f = \nabla^j f + \nabla ^c f, $$
where $\nabla^j f := \nabla^s f\llcorner \mathcal{J}_f$ and $\nabla^c f$ vanishes on $\mathcal H^{d-1}$-finite sets. Similarly, if $f \in BD(\Omega)$, let $E f = E^a f + E^s f$ be the Radon--Nikodym decomposition of $E f$ with respect to the Lebesgue measure. Then $E^s f$ admits the following decomposition in \quotes{jump part} and \quotes{Cantor part} 
$$ E^s f = E^j f + E^c f, $$
where $E^j f := E^s f\llcorner \mathcal{J}_f$ and $E^c f$ vanishes on $\mathcal H^{d-1}$-finite sets. 
\end{theorem}

We define special functions of bounded variation/deformation as follows.

\begin{definition} \label{d: SBV - SBD}
Let $\Omega \subset \R^d$ be an open set and $f \in BV(\Omega; \R^m)$. We say that $f$ is a special function of bounded variation ($f\in SBV(\Omega))$ if $\nabla^c f = 0$, i.e. $\nabla^s f = \nabla^j f$. Similarly, given $f \in BD(\Omega; \R^d)$, we say that $f$ is a special function of bounded deformation if $E^c f = 0$, i.e. $E^s f = E^j f$.  
\end{definition}

We recall that functions of bounded variation/deformation have a trace on Lipschitz hypersurfaces. We refer to \cite{AFP00}*{Theorem 3.77} for the $BV$ case and \cite{T83}*{Chapter 2}, \cite{ACDalM97}*{Section 3} for the $BD$ case.  

\begin{theorem}[Boundary trace]\label{T:trace_in_BV-BD}
Let $\Omega \subset \R^d$ be an open set and let $f \in BV(\Omega; \R^m)$ or $f \in BD(\Omega, \R^d)$. Then, for any countably $\mathcal{H}^{d-1}$-rectifiable set $\Sigma \subset \Omega$ oriented by $n_\Sigma$, then $f$ has one-sided Lebesgue limits $f^{\Sigma_\pm}$ on both sides of $\Sigma$ (with respect to $n_\Sigma$). Moreover, if $f \in BV(\Omega; \R^m)$ it holds 
$$\nabla f\llcorner \Sigma  = \left(f^{\Sigma_+} - f^{\Sigma_-}\right) \otimes n_{\Sigma} \mathcal{H}^{d-1} \llcorner \Sigma. $$
If $f \in BD(\Omega; \R^d)$ it holds 
$$ (E f) \llcorner \Sigma =  \frac{1}{2} \left[ \left(f^{\Sigma_+} - f^{\Sigma_-} \right) \otimes n_\Sigma + n_\Sigma \otimes \left(f^{\Sigma_+}- f^{\Sigma_-} \right) \right] \mathcal{H}^{d-1} \llcorner \Sigma. $$
\end{theorem}

\subsection{Time-dependent functions of bounded variation/deformation}

We use the notation from \cref{ss: curves of measures} to define time dependent functions of bounded variation/deformation. 

\begin{definition}
Let  $\Omega \subset \R^d$ be an open set. We say that $f \in L^1((0,T); BV(\Omega))$ if $ \in L^1_{x,t}$, $f_t \in BV(\Omega)$ for a.e. $t \in (0,T)$ and $\{\nabla f_t\}_{t \in (0,T)} \in L^1_t \mathcal{M}_x$. Similarly, we say that $f \in L^1((0,T); BD(\Omega))$ if $f \in L^1_{x,t}$, $f_t \in BD(\Omega)$ for a.e. $t \in (0,T)$ and $\{ E f_t\}_{t \in (0,T)} \in L^1_t \mathcal{M}_x$\footnote{Given $f \in L^1_{x,t}$ such that $\nabla f_t \in \mathcal{M}_x$ for a.e. $t$, the family $\{\nabla f_t\}_{t}$ is always weakly measurable according to \cref{d: curve of measures} by Fubini's theorem. Here, we only require that $\norm{\nabla f_t}_{\mathcal{M}_x} \in L^1_t$. The same holds for the symmetric gradient. }.
\end{definition}

Thanks to \cref{L:Bv_gradient} we can prove a time dependent decomposition result for $\div(\rho u)$, under the assumptions that $\nabla_x \rho$ is a space-time measure and $u$ is a bounded divergence-free vector field. 

\begin{lemma}\label{L:divdec}
Let $u \in L^\infty_{x,t}, \rho \in L^1_{x,t}$. Assume that $\div_x u_t = 0$ in $\mathcal{D}'(\Omega)$ for almost every $t$ and $\nabla \rho := (\partial_1 \rho, \dots, \partial_d \rho) \in \mathcal{M}_{x,t}$. Denote by $\nabla \rho = \nabla^a \rho + \nabla^s \rho$ be the Radon--Nikodym decomposition of $\nabla \rho$ with respect to $\mathcal{L}^{d+1}$. Then $\div(\rho u) \in \mathcal{M}_{x,t}$ and it holds     
\begin{align}
    \div(\rho u) = (u \cdot \nabla^a \rho) + \mu
\end{align}
with $\mu \perp \mathcal{L}^{d+1}$ and 
\begin{align} \label{eq: sing part divergence}
    \abs{\mu}  \leq \norm{u}_{L^\infty_{x,t}}  \abs{\nabla^s \rho} \qquad \text{as measures in $\Omega \times (0,T)$}. 
\end{align} 
\end{lemma}

\begin{proof}
Denote by $\rho_\eps$ the space-time mollification of $\rho$. Since $\rho_\e$ is smooth, we write
\begin{align}
    \div (\rho_{\eps} u) = (\nabla \rho)_\eps \cdot u = (\nabla^a \rho)_\eps \cdot u + (\nabla^s \rho)_\eps \cdot u.
\end{align}
For the absolutely continuous part, we know that $(\nabla^a \rho)_\eps \to \nabla^a \rho$ strongly in $L^1_{x,t}$. Since \( u \in L^\infty_{t,x} \) we infer that $ (\nabla^a \rho)_\eps \cdot u \to \nabla^a \rho \cdot u$ in $L^1_{x,t}$. By the properties of convolutions, we have that $\{(\nabla^s \rho)_\e\}_\e$ is bounded in $L^1_{x,t}$. Therefore, up to subsequence, we may assume that $(\nabla^s \rho)_\e \cdot u \rightharpoonup \mu$ weakly for some locally finite Borel measure $\mu$. Recalling that $\div (\rho_{\eps} u) \rightharpoonup \div (\rho u )$ weakly in $\mathcal D'_{x,t}$, we conclude that $\div(\rho u) = (\nabla^a \rho)\cdot u + \mu$. It remains to check \eqref{eq: sing part divergence}. For any test function $\phi \in C^\infty_c(\Omega \times (0,T))$ we have
\begin{align}
    \abs{\langle \mu, \varphi \rangle} 
    &= \abs{\lim_{\eps \to 0} \int_{\Omega \times (0,T)} (\nabla^s \rho)_\eps \cdot u \, \varphi \, \mathrm{d}x \, \mathrm{d}t} \lesssim \norm{\varphi}_{C^0_{t,x}} \limsup_{\eps \to 0} \int_{\supp (\varphi)} \abs{(\nabla^s \rho)_\eps} \, \mathrm{d}x \, \mathrm{d}t \lesssim \norm{\varphi}_{C^0_{t,x}} \abs{\nabla^s \rho} (\supp \varphi), 
\end{align}
thus proving \eqref{eq: sing part divergence}. 
\end{proof}

In the case of time dependent vector fields, it is convenient to introduce the following notion. 

\begin{definition} \label{d: BD x-t}
Given $f: \Omega \times (0,T) \to \R^d$, we say that $f \in BD_{x,t}$ ($f \in SBD_{x,t}$) if $F:= (f,1) \in BD_{x,t}$ ($F \in SBD_{x,t}$). 
\end{definition}

\begin{remark} \label{r: BD/SBD x-t}
We recall 
\begin{equation}
    E_{x,t} F = \left(
    \begin{matrix}
    E f & \frac{\partial_t f}{2}
    \\ \frac{(\partial_t f)^{T}}{2} & 0
    \end{matrix} \right), \qquad \text{where } (E f)_{i,j} = \frac{\partial_i f_j + \partial_j f_i}{2} \qquad \text{ for } i,j = 1, \dots, d.  
\end{equation}
We remark the difference between \cref{d: SBV - SBD} and \cref{d: BD x-t}. Indeed, given $f \in L^1_t BD_x$ we require that $f_t \in BD_x$ for a.e. $t$ and $t \mapsto \norm{E f_t}_{\mathcal{M}_x} \in L^1_t$. Thus, by \cref{l: curve of measure 1} the symmetric spatial gradient can be identified with a space-time measure with Lebesgue time marginal. On the other side, given $f \in BD_{x,t}$, we require that the spatial symmetric gradient and the time derivative are Radon measures, without any constraint on the time marginal. 
\end{remark}

\subsection{Measure-Divergence vector fields and traces}\label{ss: traces}

We recall the notion of distributional normal trace \cites{ACM05,Shv09,CCT19,CTZ09}. Let $\Omega \subset \R^d$ be an open set. We say that a vector field $V$ is in $\mathcal{MD}^\infty(\Omega)$ if $ V \in L^\infty(\Omega;\R^d)$ and $\div V \in \mathcal{M}(\Omega)$. For such vector fields, it is possible to define the outer normal trace on $\partial \Omega$ coherently with the Gauss--Green formula, i.e. 
\begin{equation} \label{eq: normal distributional trace}
    \langle \tr_n(V, \partial \Omega), \varphi \rangle := \int_{\Omega} \nabla \varphi \cdot V \dd x + \int_{\Omega} \varphi \dd (\div V) \qquad \forall\varphi \in C^\infty_c(\R^d).
\end{equation}
The formula above defines a distribution supported on $\partial \Omega$. We will always adopt the convention that the boundary of an open sets is oriented by the outer unit normal, denoted by $n_{\partial \Omega}$. If $V \in C^1(\overline{ \Omega};\R^d)$ and $\Omega$ has Lipschitz boundary, then $\tr_n(V, \partial \Omega)$ is induced by the integration on $\partial \Omega$ of $V \cdot n_{\partial \Omega}$. If $V \in \mathcal{MD}^\infty(\Omega)$ and $\Omega$ has Lipschitz boundary, then it turns out that $\tr_n(V, \partial \Omega)$ is induced by an $L^\infty$ function on $\partial \Omega$, still denoted by $\tr_n(V, \partial \Omega)$ (see e.g. \cite{ACM05}*{Proposition 3.2}). Moreover, this notion of trace is local in the sense that for any Borel set $\Sigma \subset \partial \Omega_1\cap \partial \Omega_2$ such that $n_{\partial \Omega_1}(x) = n_{\partial \Omega_2}(x)$ for $\mathcal{H}^{d-1}$-a.e. $x \in \Sigma$, where $\Omega_1, \Omega_2$ are Lipschitz open sets contained in $\Omega$, we have that the two traces coincide, i.e. $\tr_n(V, \partial \Omega_1) = \tr_n(V, \partial \Omega_2)$ for $\mathcal{H}^{d-1}$-a.e. $x \in \Sigma$ (see \cite{ACM05}*{Proposition 3.2}). This property allows for the definition of the distributional normal traces on an oriented Lipschitz hypersurface $\Sigma \subset \Omega$. Choose open sets $\Omega_1, \Omega_2$ such that $\Sigma \subset \partial \Omega_1 \cap \partial \Omega_2$ and $n_\Sigma(x) = n_{\partial \Omega_1}(x) = -n_{\partial \Omega_2}(x)$ for $\mathcal{H}^{d-1}$-a.e. $x \in \Sigma$. Then, we define 
$$\tr_n(V, \Sigma_-) : = \tr_n(V, \partial \Omega_1), \qquad \tr_n(V, \Sigma_+) := -\tr_n(V, \partial \Omega_2) \qquad \text{ on } \Sigma.$$
The following is one of the key properties of the distributional normal trace.

\begin{proposition}[\cite{ACM05}*{Proposition 3.4}]\label{P:div and normal trace} 
Let $V\in \mathcal{MD}^\infty(\Omega)$ and $\Sigma\subset \Omega$ be any oriented Lipschitz hypersurface. Then 
$$
\abs{\div V}(\Sigma)=\int_\Sigma \abs{\tr_n (V,\Sigma_+)- \tr_n (V,\Sigma_-)}  \dd \mathcal H^{d-1}.
$$
\end{proposition}

The notion of distributional normal trace is usually too weak to deal with non-linear problems. In particular, if $V$ has vanishing distributional normal trace and $\rho$ is a bounded scalar  function, it is not guaranteed that $\rho V$ has a distributional normal trace and, even if it does, it may not vanish. To overcome this problem, we recall from \cite{DRINnew}*{Section 2.2} a stronger notion of normal trace, which is inspired by the theory of $BV$ functions. We also refer to  \cites{DRINV23,CDRINVN24,AFP00} for an extensive presentation of the topic.

\begin{definition} [\cite{DRINnew}*{Definition 2.6}] \label{d: inner/outer trace} 
Let $\Omega \subset \R^d$ be an open set,  let $v \in L^\infty(\Omega ;\R^m)$ and let $\Sigma \subset \Omega$ be a Lipschitz hypersurface oriented by  $n_{\Sigma}$. We say that $v^{\Sigma_{\pm}} \in L^\infty(\Sigma ;\R^m)$ is the inner/outer trace on $\Sigma$ if for any sequence $r_j \to 0$ there exists a $\mathcal{H}^{d-1}$-negligible set $\mathcal{N}_{\pm} \subset \Sigma$ such that 
\begin{equation}
    \lim_{j \to \infty} \fint_{B_{r_j}^{\pm}(x)} \abs{v(z) - v^{\Sigma_{\pm}}(x)} \dd z = 0 \qquad \forall x \in \Sigma \setminus \mathcal{N}_{\pm},
\end{equation}
 where we set
\begin{equation}
    B_{r}^{\pm}(x) := \left\{ z \in B_{r}(x) \,\colon \,\langle \pm\, n_{\Sigma}(x), z-x \rangle \geq 0 \right\}. 
\end{equation}
We say that $v$ has bilateral traces on $\Sigma$ if $v$ has both inner and outer traces on $\Sigma$. Similarly, given $V \in L^\infty(\Omega; \R^d)$, we say that $V^{\Sigma_{\pm}}_n \in L^\infty(\Sigma)$ is the inner/outer Lebesgue \emph{normal} trace on $\Sigma$ if for any sequence $r_j \to 0$ there exists a $\mathcal{H}^{d-1}$-negligible set $\mathcal{N}_{\pm} \subset \Sigma$ such that 
\begin{equation}
    \lim_{j \to \infty} \fint_{B_{r_j}^{\pm}(x)} \abs{V(z)\cdot n_\Sigma(x) - V^{\Sigma_{\pm}}_n(x)} \dd z = 0 \qquad \forall x \in \Sigma \setminus \mathcal{N}_{\pm}.
\end{equation}
\end{definition}

A direct consequence is that the trace of non-linear quantities of bounded vector fields enjoy the composition formula. 

\begin{lemma} [\cite{DRINnew}*{Corollary 2.9}] \label{c: composition trace} 
Let $\Omega \subset \R^d$ be an open set, let $\Sigma \subset \Omega$ be a Lipschitz hypersurface oriented by a normal unit vector $n_{\Sigma}$ and let $n,m \in \N$. Let $V \in L^\infty(\Omega; \R^n)$ be a vector field with inner/outer trace on $\Sigma$ according to \cref{d: inner/outer trace}. Then, for $g \in C(\R^n; \R^m)$ it holds that $g(V)$ has inner/outer trace on $\Sigma$ and 
\begin{equation} \label{eq: formula composition of trace}
   g(V)^{\Sigma_{\pm}} = \pm g(V^{\Sigma_{\pm}}).  
\end{equation}
Moreover, if $n=m=d$, then the inner/outer normal Lebesgue trace of $g(V)$ on $\Sigma$ is given by 
\begin{equation} \label{eq: formula composition of normal trace}
    g(V)^{\Sigma_{\pm}}_n = \pm g(V^{\Sigma_{\pm}}) \cdot n_\Sigma. 
\end{equation}
\end{lemma}

By \cite{CDRINVN24}*{Theorem 1.4},  the inner/outer normal Lebesgue trace agrees with the distributional normal one, whenever both exist.  

\begin{theorem} \label{t: distributional trace vs lebesgue trace}
Given an open set $\Omega \subset \R^d$, let $V \in \mathcal{MD}^\infty(\Omega)$ and let $\Sigma \subset \Omega$ be a Lipschitz hypersurface oriented by $n_{\Sigma}$. Assume that $V$ has inner/outer normal Lebesgue trace on $\Sigma$ according to \cref{d: inner/outer trace}. Then, it holds 
$$ \tr_{n}(V, \Sigma_{\pm}) = V^{\Sigma_\pm}_n.$$
\end{theorem}

\section{Proofs of our results} \label{s: proofs}

\subsection{Duchon--Robert type approximation} In this section, we give a more general version of \cref{thm: DR approximation}. More precisely, we establish upper bounds on the defect distribution associated to weak solutions to \eqref{IE} in some critical classes. 

\begin{theorem} \label{thm: DR approximation sharp}
Let $(\rho,u, p) \in L^\infty_{x,t}$ be a weak solution to \eqref{IE} in $\Omega \times (0,T)$ according to \cref{d: weak solution} with force $f \in L^1_{x,t}$. Let $D[u] \in \mathcal{D}'_{x,t}$ be the Duchon--Robert distribution defined by \eqref{eq: local energy balance}. Assume that $\rho$ satisfies the no-vacuum condition \eqref{no-vacuum} and that $\rho \in L^1_t BV_x$\footnote{In the following inequalities, the implicit constants depend only on suitable norms of $\rho, u$. For a.e. $t \in (0,T)$ denote by $\nabla^s \rho_t, E^s u_t$ the singular part of $\nabla \rho_t, E u_t$ with respect to $\mathcal{L}^{d}$ and we use \cref{l: curve of measure 1} and \cref{l: curve of measure 2} to define the space-time measures $\abs{\nabla^s \rho_t} \otimes \dd t, \abs{E^s u_t} \otimes \dd t$}.  
\begin{enumerate}
    \item[(i)] If $u \in L^1_t BV_x$, then $D[u] \in \mathcal{M}_{x,t}$ and it holds 
    \begin{equation} \label{eq: upper bound DR in BV - sharp}
        \abs{D[u]} \lesssim \abs{\nabla^s \rho_t} \otimes \dd t \qquad \text{ as measure in } \Omega \times (0,T). 
    \end{equation}
    \item[(ii)] If $u \in L^1_t BD_x$, then $D[u] \in \mathcal{M}_{x,t}$ and it holds 
    \begin{equation} \label{eq: upper bound DR in BD_x - sharp}
        \abs{D[u]} \lesssim  \abs{\nabla^s \rho_t}\otimes \dd t + \abs{E^s u_t} \otimes \dd t \qquad \text{ as measure in } \Omega \times (0,T). 
    \end{equation}
    \item[(iii)] If $u \in BD_{x,t}$ in the sense of \cref{d: BD x-t}, then $D[u] \in \mathcal{M}_{x,t}$ and it holds 
    \begin{equation} \label{eq: upper bound DR in BD_x,t - sharp}
        \abs{D[u]} \lesssim \abs{\nabla^s \rho_t} \otimes \dd t + \abs{E^s_{x,t} U}\footnote{Here, $U = (u,1)$ and $E^s_{x,t} U$ is the singular part of the space-time symmetric gradient of $U$ with respect to $\mathcal{L}^{d+1}$.} \qquad \text{ as measure in } \Omega \times (0,T).  
    \end{equation}
\end{enumerate} 
\end{theorem}

We briefly comment on the assumptions in the cases $(ii)$-$(iii)$. As explained in \cref{r: BD/SBD x-t} there are no implications between $u \in L^1_t BD_x$ and $u \in BD_{x,t}$. In analogy with the homogeneous case, $L^1_t BD_x$ is a natural critical class for the regularity of the velocity field. We decided to include the case $BD_{x,t}$ to have \cref{cor: no dissipation BV/BD} a consequence of \cref{thm: DR approximation sharp} and \cref{thm: no dissipation} (see \cref{r: space-time BD BV}).  The proofs in the cases $(ii)$-$(iii)$ are rather similar, the only difference being that if $u \in L^1_t BD_x$ we are allowed to fix the time to exploit the spatial regularity, while for $u \in BD_{x,t}$ all the estimates are in space-time. In the stationary case, the approaches in $(ii)$-$(iii)$ coincide.

\begin{proof} [Proof of \cref{thm: DR approximation sharp}] 
Let $O$ be an open set such that $ O \subset \joinrel \subset \Omega \times (0,T)$. From now on, we neglect multiplicative constants depending only on suitable norms of $\rho, u,p$ in $\Omega$. Therefore, all the estimates below are uniform with respect to $O$. 

\textsc{\underline{Step 1}: Expanding the error terms.} We denote by 
\begin{equation} \label{eq: admissible kernel}
    \mathcal{K} := \left\{ \eta \in C^\infty_c(B_1) \colon \int_{B_1} \eta = 1, \ \eta \geq 0, \ \eta \text{ is even} \right\} 
\end{equation}
and we fix a convolution kernel $\eta \in \mathcal{K}$. Given $g \in L^1_{x,t}$, for $\e$ small enough we denote by $g_\e(x,t)$ the mollification of $g_t$ with respect to the spatial variable, which is well defined for $(x,t) \in O$. Given $\varphi \in C^\infty_c(O)$, we test the momentum equation in \eqref{IE} with 
\begin{align}\label{eq: testF}
   \psi_\e : = \varphi \frac{(\rho u)_\eps }{\rho_\eps}. 
\end{align}
We check that $(\rho u)_\e, \rho_\e$ are Lipschitz continuous in $O$ and $\rho_\e(x,t) \geq C >0$ for any $(x, t) \in O$. The latter follows by \eqref{no-vacuum} and the fact that $\eta$ is nonnegative. By the properties of convolution, $\rho_\e, (\rho u)_\e$ are smooth with respect to the spatial variable. Then, mollifying the transport equation and the momentum equation in \eqref{IE} we get that  
\begin{align}
    \pt \rho_\eps + \underbrace{\div ((\rho u)_\eps)}_{L^\infty (O)} = 0 \quad \text{and} \quad   \pt (\rho u)_\eps + \underbrace{\div ((\rho u \otimes u)_\eps) + \nabla p_\eps - f_\e}_{L^\infty(O)} = 0.
\end{align}

Thus, $\rho_\e, (\rho u)_\e$ are Lipschitz with respect to the time variable in $O$ and we conclude that $\psi_\e$ is Lipschitz continuous with compact support in $O$. Then, testing \eqref{eq: momentum} with \eqref{eq: testF}, we get
\begin{align} \label{eq: mollified transport and momentum}
    \underbrace{\int  \rho u \cdot \pt \left(\varphi \frac{(\rho u)_\eps }{\rho_\eps}\right) }_{I^1_\eps} + \underbrace{\int   \rho u \otimes u : \nabla \left(\varphi \frac{(\rho u)_\eps }{\rho_\eps}\right) }_{I^2_\eps} + \underbrace{\int   p \div \left(\varphi \frac{(\rho u)_\eps }{\rho_\eps}\right) }_{I^3_\eps} + \underbrace{ \int f \cdot \left( \varphi \frac{(\rho u)_\e}{\rho_\e} \right) }_{I_{\e}^4}=0. 
\end{align}
To study $I^1_\eps$, we use \eqref{eq: mollified transport and momentum} and we compute
\begin{align}
    I^1_\eps & = \int \rho u \cdot \pt \varphi \frac{(\rho u)_\eps }{\rho_\eps} + \int   \frac{\rho u  \varphi }{\rho_\eps} \cdot \pt (\rho u)_\eps - \int \rho u \cdot(\rho u)_\eps  \varphi \frac{\pt  \rho_\eps }{\rho^2_\eps}  \\ & = \underbrace{\int \frac{\rho \pt \varphi}{\rho_\eps}  u \cdot (\rho u)_\eps}_{I_\eps^{1,1}}  \underbrace{- \int   \frac{\rho   \varphi }{\rho_\eps}u\cdot  \div ((\rho u \otimes u)_\eps)}_{I_\eps^{1,2}} \underbrace{- \int   \frac{\rho   \varphi }{\rho_\eps} u \cdot \nabla p_\eps}_{I_\eps^{1,3}} + \underbrace{\int \frac{\rho \varphi  \div((u\rho)_\eps) }{\rho^2_\eps}  u \cdot(\rho u)_\eps}_{I^{1,4}_\eps} + \underbrace{ \int \frac{\rho u \varphi}{\rho_\e} \cdot f_\e }_{I_\e^{1,5}}.
\end{align}
To study the pressure term $I_\eps^{1,3}$, we need an integration by parts with respect to the spatial variables\footnote{This step is quite delicate and, since we need to fix the time, it seems crucial to have that $\rho \in L^1_t BV_x$. Here, the main issue is the time regularity of the spatial regularization of the pressure. }. Indeed, we use \cref{l: curve of measure 2} to compute $\abs{\nabla^s \rho} = \abs{\nabla^s \rho_t} \otimes \dd t$. Then, we infer that $(\pi)_\#\div(\rho u) \ll \mathcal{L}^1$, where $\pi: \Omega \times (0,T) \to (0,T)$ is the projection, and by \cref{l: curve of measure 1} we identify $\div(\rho u)$ with a curve of measures. Therefore, writing explicitly the integrals and recalling that $p_\e, \rho_\e$ are smooth with respect to the spatial variables, we have 
\begin{align}
    I_\eps^{1,3}& = - \int_0^T \int_\Omega \rho u \cdot \nabla \left(\frac{\varphi p_\eps}{\rho_\eps}\right) \dd x \dd t + \int_0^T \int_\Omega p_\eps \rho u \cdot \nabla \left( \frac{\varphi}{\rho_\eps} \right) \dd x \dd t  
    \\ & = \underbrace{ \int_0^T \int_\Omega \frac{\varphi p_\eps}{\rho_\eps} \dd (\div(\rho u))_t \dd t }_{I^{1,3,1}_\eps} + \underbrace{\int_0^T \int_\Omega p_\eps \frac{\rho u}{\rho_\eps} \cdot \nabla \varphi \dd x \dd t}_{I^{1,3,2}_\eps}- \underbrace{\int_0^T \int_\Omega p_\eps \frac{\rho}{\rho_\eps^2} \varphi u \cdot \nabla \rho_\eps \dd x \dd t }_{I^{1,3,3}_\eps}. 
\end{align}
For $I_\eps^2, I_\e^3$, by the chain rule, we have 
\begin{align}
    I_\eps^2 & = \underbrace{\int \frac{\rho}{\rho_\eps} u \otimes u : \left(\nabla \varphi \otimes (\rho u)_\eps  \right)}_{I_\eps^{2,1}} \underbrace{- \int \frac{\rho \varphi}{\rho_\eps^2} u \otimes u :  \left( \nabla \rho_\e \otimes (\rho u)_\eps \right)}_{I_\eps^{2,2}} + \underbrace{\int \frac{\rho \varphi}{\rho_\eps} u \otimes u : \nabla \left( (\rho u)_\eps \right)}_{I_\eps^{2,3}}
    \\  I_\eps^3 & = \underbrace{\int \frac{p}{\rho_\eps} \nabla \varphi \cdot (\rho u)_\eps}_{I_\eps^{3,1}} + \underbrace{\int \frac{p\varphi}{\rho_\eps} \div\left( (\rho u)_\eps\right)}_{I_\eps^{3,2}} - \underbrace{\int\frac{p\varphi}{\rho_\eps^2}  (\rho u)_\eps   \cdot \nabla \rho_\eps}_{I_\eps^{3,3}}
\end{align}
Since $f \in L^1_{x,t}, \rho, u, p \in L^\infty_{x,t}$ and $\rho$ satisfies \eqref{no-vacuum}, then $I_\eps^{1,1}+I_\eps^{2,1}+I_\eps^{3,1}+I_\eps^{1,3,2} + I^4_\e + I^{1,5}_\e$ converge to 
\begin{align}
    \int\rho  \abs{u}^2 \pt \varphi + \int \rho \abs{u}^2 u\cdot \nabla \varphi + \int 2 p u \cdot \nabla \varphi + \int 2 \varphi f \cdot u.  
\end{align}
We organize the other terms as 
\begin{align}
    & \langle J^1_{\eps}, \varphi \rangle := I_\eps^{1,2}+I_\eps^{2,3} = \int_0^T \int_\Omega  \frac{\rho   \varphi }{\rho_\eps} \left[  u \otimes u : \nabla \left( (\rho u)_\eps \right) - u\cdot  \div ((\rho u \otimes u)_\eps)\right] \dd x \dd t, \\
    & \langle J^2_{\eps},\varphi \rangle := I_\eps^{1,4} + I_\eps^{2,2} = \int_0^T \int_\Omega \frac{\rho \varphi   }{\rho^2_\eps}  u \cdot(\rho u)_\eps \left[ \div((u\rho)_\eps) - u \cdot \nabla \rho_\eps \right] \dd x \dd t, \\
    & \langle J^3_{\eps},\varphi \rangle := I_\eps^{3,2}+I_\eps^{3,3} = \int_0^T \int_\Omega \frac{p\varphi}{\rho^2_\eps} \left[ \rho_\eps \div\left(  (\rho u)_\eps\right) -  (\rho u)_\eps   \cdot \nabla \rho_\eps \right] \dd x \dd t, \\
    & \langle J^4_{\eps},\varphi \rangle := I_\eps^{1,3,1}+I_\eps^{1,3,3} = \int_0^T \left(\int_\Omega \frac{\varphi p_\eps}{\rho_\eps} \dd (\div(\rho u)_t) - \int_\Omega \frac{\phi p_\e}{\rho_\e^2} \rho u \cdot \nabla \rho_\e \dd x \right) \dd t .
\end{align}
Therefore, we have shown that the following equality in the sense of distributions
\begin{align}
    \pt \left( \rho \abs{u}^2 \right) + \div \left( u  \left(\rho \abs{u}^2 + 2p \right) \right) -2 f \cdot u = - 2 D[u] = -\lim_{\eps \to 0} \left( J^1_{\eps} + J^2_{\eps}+J^3_{\eps}+J^4_{\eps} \right).
\end{align}

\textsc{\underline{Step 2}: The terms concerning the density.} We claim that for every compact set $Q \subset O$ it holds\footnote{Here, the implicit constant depends only on norms of $\rho, u, p$ and it is uniform with respect to $Q, O$ and $\eta \in \mathcal{K}$. }
\begin{align} \label{eq: upper bound J^2,3,4}
    \limsup_{\eps \to 0}\norm{J_\eps^2+J_\eps^3+J^4_\eps}_{L^1(Q)}\lesssim [\abs{\nabla^s \rho_t} \otimes \dd t] (Q). 
\end{align}
We study $J^2_\eps$ and we leave the terms $J^3_\eps, J^4_\e$ to the reader, since it can be treated in the same way. By \cref{l: curve of measure 2} and \cref{L:divdec}, for a.e. $t$ we write 
$$\div(\rho u)_t = u_t \cdot \nabla^a \rho_t + \mu_t, $$
with $\abs{\mu_t} \lesssim \abs{\nabla^s \rho_t}$ \footnote{Here, we denote by $\nabla \rho_t = \nabla^a \rho_t + \nabla^s \rho_t$ the Radon-Nikodym decomposition of $\nabla \rho_t$ with respect to $\mathcal{L}^d$}. Then, letting $Q_t := \{ x \in \R^d \colon (x,t) \in Q \}$, for a.e. $t$ we have 
\begin{align}
   \norm{J^2_\eps(\cdot, t)}_{L^1(Q_t)}  & \lesssim \norm{\frac{\rho}{\rho^2_\eps}  u \cdot(\rho u)_\eps }_{L^\infty(Q)} \norm{  (\div \rho_t u_t )_\eps  - u_t \cdot (\nabla \rho_t)_\eps  }_{L^1(Q_t)} \\ 
   & \lesssim \norm{ (u_t \cdot \nabla^a \rho_t)_\e - u_t \cdot (\nabla^a \rho_t)_\eps}_{L^1(Q_t)} + \norm{(\mu_t)_\e  - u_t \cdot (\nabla^s \rho_t)_\eps}_{L^1(Q_t)} = : A_\e(t) + B_\e(t). 
\end{align}
Since $u_t\in L^\infty_{x}, \nabla^a \rho_t \in L^1_{t}$, we get $A_\e(t) \to 0$ as $\e \to 0$. For the singular parts, by the properties of convolution, we have 
\begin{align}
     \limsup_{\eps \to 0} B_\e(t)  & \lesssim \limsup_{\e \to 0} \norm{(\mu_t)_\e}_{L^1(Q_t)} +  \norm{(\nabla^s \rho_t)_\e}_{L^1(Q_t)} \lesssim \abs{\mu_t}(Q_t) + \abs{\nabla^s \rho_t}(Q_t)  \lesssim \abs{\nabla^s \rho_t}(Q_t). 
\end{align}
Setting $O_t := \{ x \in \R^d : (x,t) \in O \}$, we bound\footnote{This estimate is uniform with respect to $\e$. The implicit constant depends only on $\norm{u}_{L^\infty(O)}$. } 
\begin{align}
    A_\e(t) + B_\e(t) \lesssim \norm{(\nabla^a \rho_t)_\e}_{L^1(Q_t)} + \norm{(\nabla^s \rho_t)_\e}_{L^1(Q_t)} \lesssim \abs{\nabla \rho_t}(O_t) \in L^1_t. 
\end{align}
Then, by the dominated convergence theorem and Fatou's lemma, we infer 
\begin{equation}
    \limsup_{\e \to 0} \norm{J^2_\e}_{L^1(Q)} \lesssim \limsup_{\e \to 0} \norm{A_\e(t)}_{L^1_t} + \limsup_{\e \to 0} \norm{S_\e(t)}_{L^1_t} \lesssim \int_{0}^T \abs{\nabla^s \rho_t}(Q_t) \dd t. 
\end{equation}

\textsc{\underline{Step 3}: Velocity in $L^1_t BV_x$.} To handle $J_\e^1$, for $(x,t) \in O$ and $\e$ small enough, we write
\begin{align}\label{eq: Step4Need}
     J^1_\eps(x,t)=& \frac{\rho  }{\rho_\eps}(x,t)  \int_{B_1} \frac{1}{\e}  \left(\nabla \eta (y) \cdot \delta_{\eps y}u_t(x)\right)\left( u(x,t) \cdot (\rho u)(x+\eps y,t) \right)\dd y.
\end{align}
For any $y \in B_1$, we denote by $v^y (x,t) : = \nabla \eta(y) \cdot u(x,t)$. Then $v^y \in L^1_t BV_x$ and $\nabla_x v^y = \nabla \eta(y) \cdot \nabla_x u_t$. Then, for any compact set $Q \subset O$, by Fatou's lemma and using \eqref{BV_est_increment} with $v^y$, we estimate\footnote{Here, the implicit constant depends on $\rho,u,p$, but it is independent of $O, \eta$. \label{f: uniform constant} }
\begin{align} 
    \limsup \norm{J_\e^1}_{L^1(Q)} & \lesssim \limsup_{\eps\to 0} \int_{B_1} \left(  \int_Q \frac{1}{\e} \abs{\delta_{\e y} v^y_t(x)} \dd x \dd t \right)\dd y 
    \\ & \lesssim  \int_{B_1} \limsup_{\eps\to 0}\left(\int_Q \frac{1}{\e} \abs{\delta_{\e y} v^y_t(x)} \dd x \dd t \right)\dd y
    \\ &  \lesssim \int_{B_1} \int_0^T \abs{ \nabla \eta(y) \cdot (\nabla_x u_t) y}\left(Q_t\right) \dd t \dd y.
\end{align} 
Denoting by $\alpha = \abs{\nabla u_t} \otimes \dd t$, then by the Radon--Nikodym theorem we write $\nabla u_t \otimes \dd t = M(x,t) \dd \alpha$, where $M(x,t) \in \R^{d^2}$. Notice that, due to the divergence-free condition, we have
\begin{align}
    \trace(M(t,x))=0 \quad \text{for $\alpha$-a.e. $(x,t)$}.
\end{align} 
Then, we estimate 
\begin{equation} \label{eq: upper bound J^1}
    \limsup_{\e \to 0} \norm{J_\e^1}_{L^1(Q)} \lesssim \int_{Q} \left( \int_{B_1} \abs{\nabla \eta(y) \cdot M(x,t) y} \dd y \right) \dd \alpha \qquad \text{for any } Q \subset O \text{ compact}. 
\end{equation}
Therefore, combining \eqref{eq: upper bound J^2,3,4}, \eqref{eq: upper bound J^1} we infer that $D[u] \in \mathcal{M}(O)$ and it holds 
\begin{equation} 
    \abs{D[u]} \lesssim \abs{\nabla^s \rho_t} \otimes \dd t + \left( \int_{B_1} \abs{\nabla \eta(y) \cdot M(x,t) y} \dd y \right) \alpha  \qquad \text{ as measures in $O$}. 
\end{equation}
Hence, we find a uniform constant $C>0$ \footref{f: uniform constant}
\begin{equation}
    \left(\abs{D[u]} - C \abs{\nabla^s \rho_t \otimes \dd t} \right)^+ \leq C \left( \int_{B_1} \abs{\nabla \eta(y) \cdot M(x,t) y} \dd y \right) \alpha \qquad \text{ as measures in $O$}. 
\end{equation}
Recalling that $D[u]$ is independent of the choice of the convolution kernel $\eta \in \mathcal{K}$ given by \eqref{eq: admissible kernel}, by the so-called Alberti--Ambrosio optimization (see \cites{Ambr04, Cr09} or \cite{DRINV23}*{Proposition 2.8} for the precise statement needed here), we conclude 
\begin{equation}
    \left(\abs{D[u]} - C \abs{\nabla^s \rho}\right)^+ \leq C \left( \inf_{\eta \in \mathcal{K}} \int_{B_1} \abs{\nabla \eta(y) \cdot M(x,t) y} \dd y \right) \alpha = C \abs{\tr(M(x,t))} \alpha = 0  \qquad \text{ as measures in $O$}.  
\end{equation}
Therefore, $\abs{D}[u] - C \abs{\nabla^s \rho} \leq 0$ as measure in $O$, thus proving \eqref{eq: upper bound DR in BV - sharp}. 

\textsc{\underline{Step 4}: Velocity in $L^1_t BD_x$.} We assume $u \in L^1_t BD_x$. The main difference with the $BV$ case is that we are forced to choose radial kernels to show that $J^1_\e$ is uniformly bounded in $L^1_{x,t}$. Therefore, the Alberti--Ambrosio optimization is not available in this context. To be precise, we fix radial kernel $\eta \in \mathcal{K}$, that is $\eta(z)=h(\abs{z})$ for some smooth function $h: \R \to [0,+\infty)$. We split terms as in \textsc{Step 1}. Since the $BV$ regularity of the velocity field is not required in the study of the terms $J^2_\e, J^3_\e, J^4_\e$, we prove \eqref{eq: upper bound J^2,3,4} as in \textsc{Step 2}. Then, we focus on $J^1_\e$, where the spatial regularity of the velocity field plays a role. We claim that for any compact set $Q \subset O$ it holds that\footnote{Here, the implicit constant depends only on $\rho, u,p$ and it is uniform with respect to $O$.}
\begin{align} \label{eq: upper bound J^1 bis}
   \limsup_{\eps \to 0} \norm{J^1_\eps}_{L^1(Q)}\lesssim [\abs{E^s u_t} \otimes \dd t ](Q). 
\end{align}
Then, combining \eqref{eq: upper bound J^2,3,4} and \eqref{eq: upper bound J^1 bis}, we conclude that $D[u] \in \mathcal{M}(O)$ and \eqref{eq: upper bound DR in BD_x - sharp} holds. To show \eqref{eq: upper bound J^1 bis}, we write
\begin{align}
     J^1_\eps(x,t) = & \frac{\rho    }{\rho_\eps}(x,t)  \int_{B_1} \frac{h^\prime(\abs{y})}{\e \abs{y}}  \left( y \cdot \delta_{\eps y}u_t(x)\right)\left( u(x,t) \cdot (\rho u)(x+\eps y,t) \right)\dd y
\end{align}
and, by \cref{L:Bd_gradient}, we use Fatou's lemma to estimate
\begin{align}
\limsup_{\eps \to 0}  \norm{J^1_\eps}_{L^1(Q)} \leq  \int_0^T \limsup_{\eps \to 0}\norm{J^1_\eps(t)}_{L^1_x (Q_t)} \dd t.
\end{align}
Since $\rho \in L^\infty_{x,t}$ and it satisfies \eqref{no-vacuum}, we have
\begin{align}
     \norm{J_\eps^1(t)}_{L^1_x(Q_t)} \lesssim \norm{Z_\eps^{a}(t)}_{L^1_x(Q_t)} + \norm{Z_\eps^s(t)}_{L^1_x(Q_t)},
\end{align}
where we use the decomposition of the difference quotient given by \cref{L:Bd_gradient} and we set
\begin{align}
    Z^{a,s}_\eps (x,t)= \int_{B_1} \frac{h^\prime(\abs{y})}{\e \abs{y}}  \left( y \cdot \delta^{a,s}_{\eps y}u(x,t)\right)\left( u(x,t) \cdot (\rho u)(x+\eps y,t) \right)\dd y.
\end{align}
For the absolutely continuous part, by \eqref{eq:DQstrongconv in BD} and the continuity of translations, we have that 
\begin{align}
    \norm{  \e^{-1} y \cdot \delta^a_{\eps y}u_t  -   E^a u_t : y \otimes y }_{L^1_x(Q_t)} \to 0, \qquad \norm{(\rho u)(\cdot+\eps y, t) - \rho u(\cdot,t)}_{L^1_x(Q_t)} \to 0 \qquad \forall y \in B_1. 
\end{align}
Then, by \cref{L:Bd_gradient}, for a.e. $t$ we conclude that $Z_\eps^{a}(\cdot,t)$ converges in $L^1(Q_t)$ to 
\begin{align}\label{eq: convtoZ}
     \int_{B_1}   \frac{h^\prime(\abs{y})}{\abs{y}}  \left(  E^a u(x,t) : y \otimes y \right)\left( \left(\rho \abs{u}^2 \right) (x,t) \right)\dd y. 
\end{align}
The latter is equal to zero due to the divergence-free condition, since we compute 
\begin{align}
E^a u(x,t) : \int_{B_1}  h^\prime(\abs{y})  \left( \frac{  y \otimes y }{\abs{y}}\right)\dd y & =
     E^a u(x,t) : \int_{B_1}  \nabla \eta(y) \otimes y\dd y =  \trace (E^a u_t(x))  = 0.
\end{align}
For the singular part, since $\rho, u \in L^\infty_{x,t}$, we write
\begin{align}
       \norm{Z^s_\eps(t)}_{L^1_x(Q_t)}& \lesssim \int_{B_1} \frac{\abs{h^\prime(y)}}{\e \abs{y}}  \norm{ y \cdot \delta_{\eps y}^s u_t}_{L^1_x(Q_t)}\dd y 
\end{align}
and, by \eqref{eq:DQweakb in BD}, we infer that
\begin{align}
      \limsup_{\eps \to 0} \norm{Z^{s}_\eps(t)}_{L^1_x(Q_t)} & \lesssim \int_{B_1}\frac{\abs{h^\prime(y)}}{\abs{y}} \limsup_{\eps \to 0} \norm{ \e^{-1} y \cdot \delta_{\eps y}^s u_t}_{L^1_x(Q_t)}\dd y \lesssim  \abs{E^s u_t}(Q_t)
\end{align}
for a.e. $t$. Thus, by \eqref{eq: DQuniform bound in BD} and Fatou's lemma, we conclude that \eqref{eq: upper bound J^1 bis} holds. 

\textsc{\underline{Step 5}: Velocity in $ BD_{x,t}$.}
In the case \( u \in BD_{x,t} \), similar to \eqref{eq: upper bound J^1 bis}, the upper bounds for the terms $J^2_\e, J^3_\e, J^4_\e$ \eqref{eq: upper bound J^2,3,4} still holds. Therefore, it is enough to prove that for any compact $Q \subset O$ it holds
\begin{equation} \label{eq: upper bound J^1 tris}
    \limsup_{\eps \to 0} \norm{J^1_\eps}_{L^1(Q)} \lesssim \abs{E^s_{x,t} U}(Q).
\end{equation}
Then, by \eqref{eq: upper bound J^2,3,4}, we infer that $D[u] \in \mathcal{M}(O)$ and \eqref{eq: upper bound DR in BD_x,t - sharp} holds. To check \eqref{eq: upper bound J^1 tris}, as in the \textsc{Step 4}, we can write 
\begin{align}
    J^1_\eps(x,t) & = \int_{B_1} h^\prime(\abs{y}) \frac{(y,0) \cdot \delta_{(\eps y, 0)} U(x,t)}{\eps \abs{y}} (u(x,t)\cdot (\rho u)(x+ \e y, t)) \dd y \\  
    & = \int_{B_1} h^\prime(\abs{y}) \frac{(y,0) \cdot \delta^a_{(\eps y, 0)} U(x,t)}{\eps \abs{y}} (u(x,t)\cdot (\rho u)(x+\e y, t)) \dd y 
    \\ & \qquad + \int_{B_1} h^\prime(\abs{y}) \frac{(y,0) \cdot \delta^s_{(\eps y, 0)} U(x,t)}{\eps \abs{y}} (u(x,t) \cdot (\rho u)(x+ \e y, t)) \dd y = Z^a_{\e} (x,t)+ Z^s_\e(x,t), 
\end{align}
where we used that $U \in BD_{x,t}$ and the decomposition of $\delta_{(\e y, 0)} U$ given by \cref{L:Bd_gradient}\footnote{This step is quite delicate. Indeed, we write $\delta_{\e y} u$ as a space-time increment of $U$ so that we can use the space-time $BD$ regularity to decompose the difference quotient.}. Since we have 
$$\norm{\rho u (\cdot + \e y, \cdot )- \rho u}_{L^1(O)} \to 0 \qquad \forall y \in B_1, $$
we infer\footnote{The only difference with the computation in \textsc{Step 4} is that here we are not allowed to fix the time and all the estimates are in time-space.} 
\begin{align}
    Z^a_\e(x,t) \to & \int_{B_1} h^\prime(\abs{y}) \frac{E_{x,t}^a U (x,t) : (y,0) \otimes (y,0)}{\abs{y}} \left( \left( \rho \abs{u}^2\right) (x,t) \right) \dd y
    \\ & = \left( \left( \rho \abs{u}^2\right) (x,t) \right)  E^a u (x,t) : \int_{B_1} h^\prime(\abs{y}) \frac{y \otimes y}{\abs{y}} \dd y 
    \\ & = \left( \left( \rho \abs{u}^2\right) (x,t) \right) = \trace(E^a u(x,t)). 
\end{align}
Since $u_t$ is divergence-free for a.e. $t$, then $\div u = 0$ in $\mathcal{D}'_{x,t}$ and $\trace(E u) = \div u = 0 $ in $\mathcal{D}'_{x,t}$. Then, the latter vanishes almost everywhere in $O$.  For the singular part, again by \cref{L:Bd_gradient}, we have 
\begin{align}
    \limsup_{\eps \to 0} \norm{J^1_\eps}_{L^1(O)} 
    \lesssim \int_{B_1} \frac{\abs{h'(y)}}{\abs{y}} \limsup_{\eps \to 0}  \norm{\e^{-1} (y,0) \cdot \delta^s_{(\eps y, 0)} U}_{L_{t,x}^1(O)} \dd y  
    \lesssim \abs{E_{x,t}^s U}(O).
\end{align}
\end{proof}

\subsection{Dissipation on codimension 1 singular structures} 
In this section, we discuss the proof of \cref{thm: no dissipation}. This is based on the theory of measure-divergence vector field possessing bilateral traces on codimension $1$ Lipschitz hypersurfaces. 

\begin{proof} [Proof of \cref{thm: no dissipation}]
Let $\Gamma$ be a countably $\mathcal{H}^d$-rectifiable set, i.e. $\Gamma = \mathcal{N} \cup \bigcup_{i} \Sigma_i$, where $\mathcal{H}^d(\mathcal{N}) = 0$ and $\Sigma_i$ are (at most countably many) oriented Lipschitz hypersurfaces. Letting $V := \left( u\left( \rho \frac{\abs{u}^2}{2} + p \right), \rho\frac{\abs{u}^2}{2} \right) \in L^\infty_{x,t}$ and recalling that $f \in L^1_{x,t}$, we have that $V \in \mathcal{MD}^\infty_{x,t}$. Then, by $\abs{D[u]}(\mathcal{N}) =0$ (see e.g. \cite{DDI23}*{Theorem 3.1}) and, by \cref{P:div and normal trace}, for any space-time Lipschitz oriented hypersurface $\Sigma$ it holds  
\begin{equation}
\abs{D[u]}(\Sigma) =  \abs{- \div_{x,t} V + f \cdot u } (\Sigma)  = \abs{\div_{x,t}(V)}  (\Sigma) = \int_{\Sigma}   \abs{\tr_n(V, \Sigma_+) - \tr_n(V, \Sigma_-)} \dd \mathcal{H}^{d-1}, 
\end{equation} 
where $\tr_n(V, \Sigma_{\pm})$ is the inner/outer distributional normal trace of $V$ at $\Sigma$. Then, by \cref{P:div and normal trace} and \cref{t: distributional trace vs lebesgue trace}, we have  
\begin{equation} 
    \abs{D[u]}(\Sigma) = 0 \qquad \Longleftrightarrow \qquad V^{\Sigma_+}_n(y) =  V^{\Sigma_-}_n(y) \qquad \text{for $\mathcal{H}^d$-a.e. } y \in \Sigma. 
\end{equation} 
For simplicity, we denote by $n_{\Sigma} = : (n_x, n_t) \in \R^d \times \R$ and $ \rho^{\Sigma_{\pm}} = : \rho^{\pm}, u^{\Sigma_{\pm}} = : u^\pm, p^{\Sigma_{\pm}} = : p^{\pm}$. By \cref{c: composition trace}, $V$ has a Lebesgue inner/outer normal trace at $\Sigma$ and it holds 
\begin{equation}
    V^{\Sigma_\pm} =  u^{\pm } \cdot n_x \left( \rho^{\pm} \frac{\abs{u^\pm }^2}{2} + p^{\pm} \right) + \rho^{\pm} \frac{\abs{u^\pm}^2}{2} n_t. \label{eq: trace V} 
\end{equation} 
Therefore, we have to check 
\begin{equation} \label{eq: trace does not jump}
    \rho^+ ( u^{+} \cdot n_x + n_t) \frac{\abs{u^+ }^2}{2} + u^+ \cdot n_x p^+ = \rho^- ( u^{-} \cdot n_x + n_t) \frac{\abs{u^- }^2}{2} + u^- \cdot n_x p^- \qquad \text{for $\mathcal{H}^d$-a.e. } y \in \Sigma. 
\end{equation}
Since $(\rho,u,p)$ solves \eqref{IE} with $L^1$ force, we set $W := ( \rho u \otimes u + \textrm{Id } p, \rho u ) $ and, by \cref{P:div and normal trace} it holds 
\begin{equation}
    \int_{\Sigma} \abs{ \tr_n(W, \Sigma_+) - \tr_n(W, \Sigma_-) } \dd \mathcal{H}^{d} = \abs{ \div_{x,t} (W) } ( \Sigma ) = \int_{\Sigma} \abs{f} \dd  \mathcal{H}^{d} = 0. 
\end{equation}
Hence, we infer $\tr_n(W, \Sigma_+) =  \tr_n(W, \Sigma_-) $ and, by \cref{c: composition trace} and \cref{t: distributional trace vs lebesgue trace}, we conclude
\begin{equation} \label{eq: euler trace 1}
    \rho^+ u^+ (u^+ \cdot n_x) + p^+ n_x + \rho^+ u^+ n_t =  W^{\Sigma_+}_n = W^{\Sigma_-}_n =  \rho^- u^- (u^- \cdot n_x) + p^- n_x + \rho^- u^- n_t \qquad \text{for $\mathcal{H}^d$-a.e. } y \in \Sigma.
\end{equation}
Similarly, by \cref{c: composition trace} and \cref{t: distributional trace vs lebesgue trace}, since the space-time divergence of both $U := (u,1), Z := (u\rho, \rho)$ vanishes, by \cref{P:div and normal trace} we infer 
\begin{equation}
    u^+ \cdot n_x + n_t = U^{\Sigma_+}_n = U^{\Sigma_-}_n = u^- \cdot n_x + n_t \qquad \text{for $\mathcal{H}^d$-a.e. } y \in \Sigma, 
\end{equation}
\begin{equation} \label{eq: euler trace 2.0}
    \rho^+ (u^+ \cdot n_x + n_t) = Z^{\Sigma_+}_n = Z^{\Sigma_-}_n = \rho^- ( u^- \cdot n_x + n_t) \qquad \text{for $\mathcal{H}^d$-a.e. } y \in \Sigma. 
\end{equation}
In particular, we have
\begin{equation} \label{eq: euler trace 2}
    u^+ \cdot n_x = u^- \cdot n_x \qquad \text{for $\mathcal{H}^d$-a.e. } y \in \Sigma, 
\end{equation}
\begin{equation} \label{eq: euler trace 2.1}
    \rho^+ \neq \rho^- \qquad \Longrightarrow \qquad u^+ \cdot n_x + n_t = 0  \qquad \text{for $\mathcal{H}^d$-a.e. } y \in \Sigma. 
\end{equation}
Taking the scalar product of \eqref{eq: euler trace 1} by $n_x$ and using \eqref{eq: euler trace 2.0} and \eqref{eq: euler trace 2}, we find 
\begin{equation} \label{eq: euler trace 3}
    p^+ \abs{n_x}^2 = p^- \abs{n_x}^2 \qquad \text{for $\mathcal{H}^d$-a.e. } y \in \Sigma. 
\end{equation}
To conclude, we consider different cases. If $n_x = 0$, then $n_t = \pm 1$ and \eqref{eq: trace does not jump} is satisfied if and only 
\begin{equation} \label{eq: euler trace 2.2}
    \rho^+ \abs{u^+}^2 = \rho^- \abs{u^-}^2. 
\end{equation}
However, if $n_x = 0$, then by \eqref{eq: euler trace 1} we infer $\rho^+ u^+ = \rho^- u^-$ and by \eqref{eq: euler trace 2.0} we have $\rho^+ = \rho^-$. Therefore, \eqref{eq: trace does not jump} is satisfied. From now on, we assume $n_x \neq 0$. In this case, by \eqref{eq: euler trace 3} we infer $p^+ = p^-$ for $\mathcal{H}^d$-a.e. $y \in \Sigma$ such that $n_x(y) \neq 0$. If $\rho^+ \neq \rho^-$, then \eqref{eq: trace does not jump} follows by \eqref{eq: euler trace 2} and \eqref{eq: euler trace 2.1} . On the other side, if $\rho^+ = \rho^- >0$ (as follows by \eqref{no-vacuum}), by \eqref{eq: euler trace 1} we get that $\eqref{eq: trace does not jump}$ is satisfied. 
\end{proof}

\begin{remark} \label{r:final comments stationary}
By the analysis of the normal traces, we gain insight into the physical mechanism that prohibits dissipation along hypersurfaces. In the stationary case, fix a space hypersurface $\Sigma$. Particles that intersect $\Sigma$ transversally, i.e., with $u \cdot n \neq 0$, do not experience an instantaneous impulse, as shown by \eqref{eq: euler trace 1}, since they transport density across $\Sigma$ (see \eqref{eq: euler trace 2.1}), and the hydrodynamic pressure remains continuous across $\Sigma$ (see \eqref{eq: euler trace 3}). Indeed, the velocity is allowed to jump when particles approach $\Sigma$ with $u \cdot n = 0$. In this case, by \eqref{eq: trace does not jump} and due to the incompressibility of the velocity field (see \eqref{eq: euler trace 2}), the jump occurs only in the tangential component and it is evident that no dissipation occurs at $\Sigma$. The situation is different in a compressible model such as 1D Burgers, where dissipation occurs due to the instantaneous drop in the velocity of particles at the shock, resulting from inelastic collisions. From our investigation, we find a substantial qualitative difference between the dissipation mechanisms in the incompressible Euler equations (even in the inhomogeneous case) and in the 1D Burgers equation.
\end{remark}

\begin{remark} \label{r: final comments 2 fluids}
By the analysis of the normal traces we have a description of the dynamics in the case of a system made by two immiscible fluids. Assume that $\rho = \rho_1 \mathds{1}_A + \rho_2 \mathds{1}_B$ for some $\rho_1 \neq \rho_2$ nonnegative and $A, B$ disjoint open sets where $\overline{A} \cup \overline{B} = \Omega \times (0,T)$, and the jump set $\Sigma$ is a Lipschitz oriented space-time hypersurface. Under the assumptions of \cref{thm: no dissipation}, then by the conservation of mass, i.e. the analysis of the normal traces for the continuity equation, it follows $n_t = - u \cdot n_x$ along $\Sigma$. Hence, the time slice $\Sigma_t$ is pushed by the velocity field according to the formula above. In the stationary case, the velocity is tangent to a space hypersurface where the density exhibits a jump, coherently with the fact that the hypersurface is stationary. Lastly, we conclude that the dissipation is not concentrated on the jump set $\Sigma$ of the density because the velocity is always tangential to $\Sigma$ (see \cref{r:final comments stationary}).  
\end{remark}

\begin{proof} [Proof of \cref{cor: no dissipation BV/BD}] 
To begin, we recall that $\rho \in BV_{x,t}$ implies $\nabla \rho \in \mathcal{M}_{x,t}$. If $u \in L^1_t BV_x$, then $D[u]$ is a finite Borel measure by \cref{thm: DR approximation sharp} and, since $\rho \in SBV_{x,t}$, we find a countably $\mathcal{H}^d$-rectifiable set $\Sigma$ such that $D[u] = D[u] \llcorner \Sigma$ (see \cref{d: SBV - SBD}). Since $p$ has bilateral traces at any space-time oriented Lipschitz hypersurface by assumptions and the same holds for $\rho, u$ by \cref{T:trace_in_BV-BD}, then $\abs{D[u]}(\Sigma) = 0$ follows by \cref{thm: no dissipation}. Therefore, we have $D[u]\equiv 0$ in $\mathcal{M}_{x,t}$. If $u \in SBD_{x,t}$ in the sense of \cref{d: BD x-t}, the argument is the same, since the support of $E^s_{x,t} U$ is countably $\mathcal{H}^d$-rectifiable. 
\end{proof}

\begin{proof} [Proof of \cref{cor: global conservation}]
Fix $t \in (0,T)$ such that $u_t$ has zero Lebesgue normal boundary trace at $\partial \Omega$ according to \cref{d: inner/outer trace}. Since $\rho, u, p \in L^\infty_{x,t}$, by \cref{c: composition trace} we infer that $Z_t := \left( \frac{\rho_t \abs{u_t}^2}{2} + p_t \right) u_t$ has zero Lebesgue normal boundary trace at $\partial \Omega$. Then, by \cite{DRINV23}*{Proposition 2.3}, we have  
\begin{equation} \label{eq: limit Z to 0}
    \lim_{\e \to 0} \int_{\Omega} \abs{Z_t \cdot \nabla \chi_\e}\dd x = 0, \qquad \text{where } \chi_\e(x) : = \min \{ 1, \e^{-1} d_{\partial \Omega}(x) \}. 
\end{equation}
Since $\chi_\e \in W^{1,\infty}_0(\Omega)$, by a standard approximation argument we infer that $\psi(x,t) := \alpha(t) \chi_\e(x)$ can be used as a test function in \eqref{eq: local energy balance} for any $\alpha \in C^\infty_c((0,T))$. Thus, for any $\e>0$, we have  
\begin{equation}
    \underbrace{\int_0^T \left( \int_{\Omega} \frac{\rho \abs{u}^2}{2} \chi_\e \dd x  \right) \partial_t \alpha \dd t}_{ I^1_\e} + \underbrace{\int_0^T \left( \int_{\Omega} Z \cdot \nabla \chi_\e \dd x \right) \alpha \dd t}_{I_\e^2} = - \underbrace{\int_{0}^T \left( \int_\Omega f \cdot u \chi_\e \dd x \right) \alpha \dd t}_{I_\e^3}. 
\end{equation}
Since $\varphi_\e \to \mathds{1}_{\Omega}$ pointwise and $\abs{\varphi_\e (x)} \leq 1$ for any $x$, by the dominated convergence theorem we conclude  
\begin{equation}
    \lim_{\e \to 0} I^1_\e = \int_0^T \left( \int_\Omega \frac{\rho \abs{u}^2}{2} \dd x \right) \partial_t \alpha \dd t, \qquad \lim_{\e \to 0} I_\e^3 = \int_0^T \left( \int_\Omega f \cdot u \dd x \right) \alpha \dd t. 
\end{equation}
To handle $I^2_\e$, by \cite{Fed}*{Theorem 3.2.39}, we find $\e_0>0$ depending only on $\partial \Omega$ such that 
\begin{equation}
    \sup_{\e < \e_0} \int_{\Omega} \abs{Z_t \cdot \nabla \varphi_\e} \dd x \leq \norm{Z_t}_{L^\infty_x} \frac{\mathcal{L}^d((\partial \Omega)_\e)}{\e} \leq 2 \norm{Z_t}_{L^\infty_x} \mathcal{H}^{d-1}(\partial \Omega) \in L^1_t. 
\end{equation}
Then, by the inequality above, \eqref{eq: limit Z to 0} and the dominated convergence theorem, we infer $I^2_\e \to 0$ as $\e \to 0$, thus proving  
\begin{equation}
    \int_0^T \left( \int_\Omega \frac{\rho \abs{u}^2}{2} \dd x \right) \partial_t \alpha \dd t = - \int_0^T \left( \int_\Omega f \cdot u \dd x \right) \alpha \dd t \qquad \forall \alpha \in C^\infty_c((0,T)).
\end{equation}
\end{proof}

\bibliographystyle{plain} 
\bibliography{biblio}

\end{document}